\documentclass[leqno]{amsart}
\usepackage{amssymb,yhmath}
\usepackage{mathrsfs}
\usepackage{amsmath, amsfonts, vmargin, enumerate}
\usepackage{color}
\usepackage{verbatim}
\usepackage{amsthm}
\usepackage{latexsym, bm,ulem}

\usepackage{euscript}
\usepackage{dsfont}

\usepackage{lmodern} %pour de plus belles polices
 \newcommand{\tens}{\otimes}
  \newcommand{\plB}[3]{\mathring{#1}_{#2_1}\otimes_{\m B}\cdots\otimes_{\m B}
 \mathring{#1}_{#2_{#3}}}
 \newcommand{\un}{\underline}%
  \newcommand{\C}{\mathbb C}% 
\newcommand{\E}{\mathbb E}%
 \newcommand{\m}{\mathcal}
  \newcommand{\M}  {\mathcal M}%
 \newcommand{\A}{\ensuremath{\mathcal{A}}}%
 \newcommand{\B}{\ensuremath{\mathcal{B}}}%
  \newcommand{\bE}{\ensuremath{\mathbb{E}}}
  \newcommand{\T}{\ensuremath{\mathbf{T}}}%
  \newcommand{\F}{\ensuremath{\mathbb{F}}}%

  \newcommand{\real}{{\mathbb R}}
\newcommand{\nat}{{\mathbb N}}
\newcommand{\ent}{{\mathbb Z}}
\newcommand{\com}{{\mathbb C}}

\renewcommand{\T}{{\mathbb T}}

\renewcommand{\a}{\alpha}

\newcommand{\g}{\gamma}
\newcommand{\G}{\Gamma}

\renewcommand{\d}{\delta}

\newcommand{\e}{\varepsilon}

\renewcommand{\l}{\lambda}

\newcommand{\ot}{\otimes}

\newcommand{\8}{\infty}
\newcommand{\el}{\ell}

\newcommand{\la}{\langle}
\newcommand{\ra}{\rangle}
\newcommand{\wt}{\widetilde}
\newcommand{\wh}{\widehat}

\newcommand{\les}{\lesssim}

\newcommand{\be}{\begin{align*}}
\newcommand{\ee}{\end{align*}}
\newcommand{\beq}{\begin{equation}}
\newcommand{\eeq}{\end{equation}}

 \newcommand{\ov}{\overset}
 
\newcommand{\bt}  {\lceil\hspace{-.18cm}\rfloor} 
\newcommand{\tb}  {\rceil\hspace{-.16cm}\lfloor}
\newcommand{\btIj}{\underset{i+} {\bt}^{j}}

\newtheorem{thm}{Theorem}[section]

\newtheorem{lemma}[thm]{Lemma}
\newtheorem{corollary}[thm]{Corollary}
\newtheorem{prop}[thm]{Proposition}

\newtheorem{rem}[thm]{Remark}
\newenvironment{rk}{\begin{rem}\rm}{\end{rem}}
\newtheorem{example}[thm]{Example}
\newenvironment{ex}{\begin{example}\rm}{\end{example}}

\newtheorem{dfn}[thm]{Definition}

\begin{document}

\title[A H\"ormander-Mikhlin multiplier theory for free groups]{A H\"ormander-Mikhlin multiplier theory for free groups and amalgamated free products of von Neumann algebras}


\author[T.~Mei]{Tao Mei}
\address{Department of Mathematics, Baylor University, Waco, TX  USA}
\email{tao\_mei@baylor.edu}

\author[\'E.~Ricard]{\'Eric Ricard}
\address{Normandie Univ, UNICAEN, CNRS, LMNO, 14000 Caen, France}
\email{eric.ricard@unicaen.fr}

\author[Q.~Xu]{Quanhua Xu}
\address{%Institute for Advanced Study in Mathematics, Harbin Institute of Technology,  Harbin 150001, China; and
Laboratoire de Math{\'e}matiques, Universit{\'e} de Bourgogne Franche-Comt{\'e}, 25030 Besan\c{c}on Cedex, France}
\email{qxu@univ-fcomte.fr}

\date{}
\begin{abstract}
We establish a platform to transfer $L_p$-completely bounded maps on tensor products of  von Neumann algebras to  $L_p$-completely bounded maps on the corresponding amalgamated free products. As a consequence,  we obtain a H\"ormander-Mikhlin multiplier theory for free products of groups. Let $\mathbb{F}_\infty$ be a free group on infinite generators $\{g_1, g_2,\cdots\}$. Given $d\ge1$ and a bounded symbol $m$ on $\mathbb{Z}^d$ satisfying  the classical  H\"ormander-Mikhlin condition, the linear map $M_m:\mathbb{C}[\mathbb{F}_\infty]\to \mathbb{C}[\mathbb{F}_\infty]$ defined by $\lambda(g)\mapsto m(k_1,\cdots, k_d)\lambda(g)$ for $g=g_{i_1}^{k_1}\cdots g_{i_n}^{k_n}\in\mathbb{F}_\infty$ in reduced form (with $k_l=0$ in $m(k_1,\cdots, k_d)$ for $l>n$), extends to a complete bounded map on $L_p(\widehat{\mathbb{F}}_\infty)$ for all $1<p<\infty$, where $\widehat{\mathbb{F}}_\infty$ is the group von Neumann algebra of $\mathbb{F}_\infty$. A similar result holds for any free product of discrete groups. 
\end{abstract}

\maketitle

%%%%%%%%%%%%%%%%%%%%%%%%%%%%%%%%%%%%%%%%%%%%%%%%%%%%%%%%%%%%%%%%%%%%%%%%
%%%%%%%%%%%%%%%%%%%%%%%%%%%%%%%%%%%%%%%%%%%%%%%%%%%%%%%%%%%%%%%%%%%%%%%%
\section{Introduction}
%%%%%%%%%%%%%%%%%%%%%%%%%%%%%%%%%%%%%%%%%%%%%%%%%%%%%%%%%%%%%%%%%%%%%%%%
%%%%%%%%%%%%%%%%%%%%%%%%%%%%%%%%%%%%%%%%%%%%%%%%%%%%%%%%%%%%%%%%%%%%%%%%

Fourier multipliers are by far the most important operators in analysis. One can say that classical harmonic analysis has been developed around Fourier multipliers. The main task there is to find criteria for the boundedness of Fourier multipliers on various function spaces, notably on $L_p$-spaces. However, it is in general impossible to characterize their $L_p$-boundedness for finite $p\neq2$. In this regard, the most fundamental result is the celebrated H\"ormander-Mikhlin multiplier theorem  which  asserts that if $m$ is a $C^{[\frac d2]+1}$ function on $\mathbb
 R^d\backslash\{0\}$ such that 
  $$\sup_{0\leq |\a|\leq [\frac d2]+1}\,\sup_{\xi\in \mathbb R^d\backslash\{0\}} |\xi|^{|\a|}|\partial^\a m(\xi)|<\8,$$  
 where $\a=(\a_1,\cdots, \a_d)$ denotes a multi-index of nonnegative integers and $|\a|=\a_1+\cdots+\a_d$, 
 then the associated Fourier multiplier $T_m$, defined via Fourier transform by $\wh{T_m(f)}=m\wh f$, is a bounded map on $L_p({\mathbb R}^d)$
for all $1<p<\infty$. A similar statement holds for the periodic case, i.e., for $\T^d$ instead of $\real^d$. More precisely, if $m:\ent^d\to\com$ satisfies
 \beq\label{HM}
 \|m\|_{{\rm HM}}=\sup_{0\leq |\a|\leq [\frac d2]+1}\,\sup_{k\in \ent^d} |k|^{|\a|}|\partial^\a m(k)|<\8,
 \eeq 
 where $\partial^\a=\partial_{k_1}^{\a_1}\cdots \partial_{k_d}^{\a_d}$ denotes the partial discrete difference operator of order $\a=(\a_1,\cdots, \a_d)$,
 then the associated Fourier multiplier $T_m$ is a bounded map on $L_p(\T^d)$ for all $1<p<\infty$.
 
\medskip

In noncommutative harmonic analysis and operator algebras, the study Fourier multipliers on general groups is of paramount importance. These multipliers are also called Herz-Schur multipliers in the literature. The underlying $L_p$-spaces are then the noncommutative $L_p$-spaces on the group von Neumann algebras. This line of investigation finds one of its origins in the study of approximation properties of operator algebras inaugurated by Haagerup's pioneering work \cite{Ha78} on Fourier multipliers on free groups in which he solved the longstanding problem on Grothendieck's approximation property for the reduced C*-algebra of a free group: this algebra has the completely bounded approximation property. Later, together with De Canni\`ere \cite{dCH85} and Cowling \cite{CH89}, Haagerup studied completely bounded multipliers on the Fourier algebras of locally compact groups.

One of the remarkable recent achievements in the interplay between multipliers and noncommutative $L_p$-spaces is the result of Lafforgue and de la Salle \cite{LadS11} asserting that the noncommutative $L_p$-space with $p>4$ associated to the von Neumann algebra of $SL_3(\ent)$ fails the completely bounded approximation property. Haagerup and coauthors \cite{HdL13, HdL16, HKdL16} extended their result to connected simple Lie groups; later, de Laat and de la Salle \cite{dLdS18} discovered that $L_p$ multipliers are very tightly connected with Banach space geometry and group representations. Other striking illustrations of this interplay are Junge and Parcet's work \cite{JP08} on the Maurey-Rosenthal factorization for noncommutative $L_p$-spaces which relies on Schur multiplier estimates. Multipliers are also intimately related to functional calculus in von Neumann algebras as shown by the resolution of Krein's celebrated Lipschitz continuity problem in noncommutative $L_p$-spaces by Potapov and Sukochev \cite{PS11}, see also \cite{CPSZ19} where the crucial use of Fourier multipliers is more apparent.  

\medskip

In a recent series of articles \cite{HM12, HSS10}, Haagerup and co-authors obtained a beautiful handy characterization of the completely bounded radial multipliers on the von Neumann algebras (i.e. $L_\8$) of free groups and $p$-adic groups, and, more generally, on homogeneous trees and free products; see also Wysocza\'{n}ski's previous work \cite{Wy95}. All this has motivated many other publications of the same nature, in particular, building on Ozawa's work \cite{Oz08}, Mei and de la Salle \cite{MdS17} obtained a similar characterization for hyperbolic groups. Nevertheless, little has been  understood for non-radial multipliers\footnote{A better understanding of non-radial multipliers seem to be  in order to answer  some fundamental open questions, such as the construction of Schauder basis, for the free group reduced $C^*$-algebras and the corresponding noncommutative $L_p$-spaces.}, nor for $L_p$-spaces with $2<p<\8$. Here, the most challenging problem is to find a H\"ormander-Mikhlin type criterion for Fourier multipliers on $L_p$-spaces to be completely bounded.
% The paper \cite{BF06} shows that radial multipliers  are not enough to solve some basic problems for free group $C^*$ algbras.}

The recent papers \cite{JMP14, JMP18, PRdS} provide a first advance towards this direction. The work \cite{MR17}  by the first and second named authors  is more related to the present article, it introduced the free Hilbert transforms on free groups and free products of von Neumann algebras and showed their $L_p$-boundedness for all $1<p<\8$. We pursue this line of investigation by going considerably beyond,  and will exhibit a large family of  $L_p$ multipliers in the free setting.

\medskip

To proceed further, we need some notation and refer to the next section for all unexplained notions. Let $\G$ be a discrete group and $\l$ its left regular representation on $\el_2(\G)$. 
The group von Neumann algebra $vN(\G)$ is the von Neumann algebra generated by $\l(\G)$, it is also equal to the weak* closure of the group algebra $\com[\G]$ that consists of all polynomials in $\l$:
 $$\com[\G]=\big\{\sum_{g\in\G} \a(g)\l(g): \a(g)\in \com\big\}.$$
Following notation in quantum groups, we will denote $vN(\G)$ by $\wh\G$ in this article. $\wh\G$ is equipped with its canonical tracial state $\tau$ defined by $\tau(x)=\la\d_e,\, x\d_e\ra$ for any $x\in\wh\G$, where $\d_e$ denotes the Dirac mass at the identity $e$ of $\G$. $L_p(\wh\G)$ is the noncommutative $L_p$-space based on $(\wh\G,\,\tau)$, it is equipped with the natural operator space structure introduced by Pisier.

\smallskip

For a complex function $m$ on $\G$, the associated Fourier multiplier $T_m$ is a linear map on $\com[\G]$ determined  by $\l(g)\mapsto m(g)\l(g)$. We call $m$  a ({\it completely}) {\it bounded $L_p$-multiplier on} $\G$ if $T_m$ extends to a (completely) bounded map on $L_p(\wh\G)$. Our main aim is to find sufficient conditions for $m$ to be such a multiplier. Note that the periodic H\"ormander-Mikhlin theorem recalled at the beginning concerns the case $\G=\ent^d$ in which $\wh\G=L_\8(\T^d)$, $\T$ being the unit circle equipped with normalized Haar measure. In this case, every function $m$ on $\ent^d$ satisfying \eqref{HM} is a completely bounded $L_p$-multiplier on $\ent^d$.

 Throughout the article, $\F_\8$ will denote a free group on infinite generators $\{g_1, g_2,\cdots\}$. An element $g\in\F_\8$ different from the identity $e$ is written as  a reduced word on $\{g_1, g_2,\cdots\}$:
   \beq\label{reduced form}
   g=g_{i_1}^{k_1}g_{i_2}^{k_2}\cdots g_{i_n}^{k_n}\;\text{ with }\; i_1\neq i_2\neq\cdots\neq i_n,\; k_j\in\ent\setminus\{0\}, \;1\le j\le n.
   \eeq
 Given a complex function $m$ on  ${\mathbb Z}^d$ and a fixed $d\in {\mathbb N}$, we define $M_m: \com[\F_\8] \to \com[\F_\8] $ by
  \begin{eqnarray*}
 M_m(\l(g))
= \left\{\begin{array}{lcl}
m(k_1,\cdots, k_d)\,\l(g) & \textrm{ if}& d\le n, \\
m(k_1,\cdots k_n, 0,\cdots, 0)\,\l(g) & \textrm{ if}& d>n\end{array}\right. 
\end{eqnarray*}
for every  $g$ as in \eqref{reduced form}.  Below is our first main theorem. It is contained in \cite{MX19} for $d=1$.

\begin{thm}\label{main1} 
  Let $d\in\nat$ and $1<p<\infty$. If $m$ is a completely bounded $L_p$-multiplier on $\ent^d$, then  $M_m$ extends to a completely bounded map on $L_p(\wh{\F}_\infty)$.
 In particular, the conclusion holds if $m$ satisfies \eqref{HM}.
   \end{thm}

To achieve Theorem \ref{main1},   we establish a new platform to   transfer the $L_p$-complete boundedness  of Fourier multipliers on  tori to  Fourier multipliers on free groups. The key tool is   a  group of unitary actions on $L_2(\wh\F_\infty)$ and its complete boundedness on $L_p(\wh\F_\infty)$ for all $1<p<\infty$. Given $z=(z_{j,i})_{1\le j\le d,\, 1\le i<\8}\in \T^d\times\T^\8$, let
 $$\a^{Ld}_{z}(\l(g))= z_{1,i_1}^{k_1}z_{2,i_2}^{k_2}\cdots z_{d,i_d}^{k_d} \l(g) $$
for  $g$ as in \eqref{reduced form} (with $k_\ell=0$ if $\ell>n$).
It is obvious that $\a^{Ld}$ extends to a unitary action of the group $\T^d\times\T^\8$ on $L_2(\wh\F_\infty)$.

\begin{thm} \label{main2}
 The above action $\a^{Ld}$ of $\T^d\times\T^\8$ on $L_2(\wh\F_\infty)$ extends to   a uniformly completely bounded action on $L_p(\wh\F_\8)$ for every $1<p<\infty$ with constant depending only on $d$ and $p$. 
 \end{thm}
 
Theorem \ref{main1} will follow from Theorem \ref{main2} and a standard transference argument.  Our initial proof of Theorem \ref{main2}  is very long and technical. We have then luckily found a new and shorter argument; the new one depends, in a crucial way, on  an analogous result on amalgamated free products of von Neumann algebras. The new approach has a more advantage that it  allows us to easily extend Theorems \ref{main1} and \ref{main2} to free products of general discrete groups too. 

\medskip

In this context, let $(\A_i,\tau_i)_{i\in I}$ be a family of finite von Neumann algebras equipped with normal faithful tracial states. Let $\B$  be a common von Neumann subalgebra of the $\A_i$'s and $\E_i:\A_i\to \B$ be the corresponding trace preserving conditional expectation. We denote by $(\A,\tau)=*_{i\in I,\B}(\A_i,\tau_i)$ the amalgamated free product of the $(\A_i, \tau_i)$'s over $\B$.  Let $\E$ be the conditional expectation from $\A$ to $\B$. As usual, for $X\subset L_p(\A)$ we write $\mathring X$ for the set $\{x-\E(x): x\in X\}$.  

Assume that  we are given maps $T_{j,i}: L_p(\A_i)\to L_p(\A_i)$  with  $1\le j\le d$ and $i\in I$ satisfying the following conditions
 \begin{itemize}
 \item[$\bullet$] $T_{j,i}$ is $\B$-bimodular, that is,  $T_{k,i}(axb)=aT_{j,i}(x)b$ for $a, b\in\B$ and $x\in  L_p(\A_i)$;
 \item[$\bullet$] $T_{j,i}\big(\widering{L_p(\A_i)}\big)\subset \widering{L_p(\A_i)}$.
 \end{itemize}
We can define a map $T^{Ld}$ as follows.   $T^{Ld}(b)=b$ for $b\in \B$ and 
 \begin{eqnarray*}
 T^{Ld}(a)
 = \left\{\begin{array}{lcl}
T_{1,i_1}(a_1) \tens \cdots\tens  T_{d,i_d}(a_d)\tens a_{d+1}\tens\cdots\tens a_n& \textrm{if}& d< n, \\
T_{1,i_1}(a_1) \tens \cdots\tens T_{n,i_n}(a_n)& \textrm{if}& d\geq n\end{array}\right. 
\end{eqnarray*}
for $a=a_1\tens\cdots\tens a_n\in \plB {\m A} i n$ with  $i_1\neq i_2\neq\cdots\neq i_n$.

\smallskip

The following theorem  and its proof contain the main novelty of the article, it  gives a surprisingly simple condition to ensure the complete
boundedness of the map $T^{Ld}$ on $L_p(\A)$.

\begin{thm}\label{main3} 
 Let $1<p<\infty$. Assume that the $T_{j,i}$'s extend to completely bounded maps on $L_p(\m A_i)$ and on $L_2(\m A_i)$ uniformly in $i,j$.
Then  $T^{Ld}$ extends to a completely bounded map on $L_p(\A)$ with
$$\|T^{Ld}\|_{{\rm cb}(L_p(\A))}\lesssim_{p,d} \prod_{j=1}^d \sup_{i\in I} \Big(\|T_{j,i}\|_{{\rm cb}(L_p(\m A_i))} +\|T_{j,i}\|_{{\rm cb}(L_2(\m A_i))}\Big).$$
\end{thm}

It is clear that the assumption above is necessary for the validity of the conclusion. Thus the theorem gives a characterization of the complete boundedness of $T^{Ld}$ on  $L_p(\A)$. Another aspect to be emphasized is the fact that $p$ is a single fixed index.  One of the  main ingredients in our argument is a length reduction formula for the $L_p$-norms associated with amalgamated free products that we state as Theorem~\ref{redform}. This formula is interesting for its won right, it was previously proved in \cite[Theorem B]{JPX07}  for homogeneous free polynomials.

\medskip

The article is organized as follows. After a brief introduction to
necessary preliminaries, we give in section~\ref{Norms on modules}
some norms on Hilbert $\B$-modules that are the modular extensions of
the usual column and row noncommutative $L_p$-norms; we show that
bounded modular maps extend to these modular noncommutative
$L_p$-spaces. The results in this section provide crucial tools for
section~\ref{Multipliers on amalgamated free products} that is the
core of the article and contains our
major ideas. Section~\ref{Multipliers on amalgamated free products}
gives the proof of Theorem~\ref{main3} that is done through a
reduction formula for polynomials in a free product of von Neumann
algebras, the latter is proved in its turn with the help of an
intermediate result (Theorem~\ref{rep}) that is a special case of
Theorem~\ref{main3} where $d=1$, $p>2$ and every
$T_{i}$ is a  $*$-representation on $\A_i$ leaving
  $\B$ invariant and commuting with $\mathbb E_i$ for all $i\in I$.
Theorems~\ref{main1} and \ref{main2} easily follow from the results in
section~\ref{Multipliers on amalgamated free products}; moreover, the
arguments can be extended to the free products of general (non
abelian) discrete groups. We do all this in section~\ref{Multipliers
  on free products of groups}. This section also contains some more paraproducts
which might be interesting in free analysis. We end the article with
an appendix on the endpoint $p=\8$.

\medskip

We will use the following convention through the article:  $A\lesssim B$ (resp. $A\lesssim_p B$) means that $A\le c B$ (resp. $A\le c_p B$) for some absolute positive constant $c$ (resp. a positive constant $c_p$ depending only on $p$).  $A\simeq B$ or  $A\simeq_p B$ means that these inequalities as well as their inverses hold.

%%%%%%%%%%%%%%%%%%%%%%%%%%%%%%%%%%%%%%%%%%%%%%%%%%%%%%%%%%%%%%%%%%%%%%%%
%%%%%%%%%%%%%%%%%%%%%%%%%%%%%%%%%%%%%%%%%%%%%%%%%%%%%%%%%%%%%%%%%%%%%%%%
\section{Preliminaries}\label{Preliminaries}
%%%%%%%%%%%%%%%%%%%%%%%%%%%%%%%%%%%%%%%%%%%%%%%%%%%%%%%%%%%%%%%%%%%%%%%%
%%%%%%%%%%%%%%%%%%%%%%%%%%%%%%%%%%%%%%%%%%%%%%%%%%%%%%%%%%%%%%%%%%%%%%%%

This section presents some preliminaries on noncommutative $L_p$-spaces and amalgamated free products. 

\medskip

Murray and von Neumann's work \cite{MvN36} demonstrates  von Neumann algebras  as  a  natural framework to do  noncommutative  analysis. The elements in a von Neumann algerba $\M$ can be integrated over the equipped trace $\tau$ and measured by the associated $L_p$-norms. In the sequel, $\M$ will denote a von
Neumann algebra with a normal faithful semifinite trace $\tau$.   For each $1\le p \leq \infty,$ let $L_p (\M)$ be the  noncommutative $L_p$-space associated with the pair $(\M,\tau)$. Note that $L_\8(\M)=\M$ with the operator norm. We refer to \cite{PX03} for details and  more historical references. 

\smallskip

All $L_p$-spaces in this article will be equipped with their natural operator space structure introduced by Pisier in \cite{pis-ast, pis-intro}. We also refer to \cite{ pis-intro} for operator space theory. 
In fact, what we will need from the latter theory is only \cite[Lemma~1.7]{pis-ast} that asserts a linear map $T: L_p(\M)\to L_p(\M)$ is completely bounded iff ${\rm Id}_{S_p}\ot T: L_p(B(\el_2)\overline\ot\M)\to L_p(B(\el_2)\overline\ot\M)$ is bounded; in this case, the completely bounded norm of $T$ is equal to the usual norm of ${\rm Id}_{S_p}\ot T$ and denoted by $\|T\|_{{\rm cb}(L_p(\M))}$ or simply $\|T\|_{{\rm cb}}$ when no ambiguity is possible. Note that $L_2(\M)$ is an operator Hilbert space and every bounded map on it is automatically completely bounded  with the same norm.

\smallskip

We will often be concerned with group von Neumann algebras. Let $\G$ be a discrete group. Recall the definition of its left regular representation $\l$: for any $g\in\G$, $\l(g)$ is just the left translation by $g$ on $\el_2(\G)$, i.e., $\l(g)(f)(h)=f(g^{-1}h)$ for any $h\in\G$ and $f\in\el_2(\G)$. Thus $\l(g)(\d_h)=\d_{gh}$, where $\d_h$ is the Dirac mass at $h$. $\com[\G]$ denotes the linear span of $\l(\G)$. Then the group von Neumann algebra $\wh\G$ is the weak* closure of $\com[\G]$ in $B(\el_2(\G))$. The canonical trace of $\wh\G$ is the vector state induced by $\d_e$ with $e$ the identity of $\G$. When $\G$ is abelian (e.g. $\G=\ent^d$)  $L_p(\wh\G)$ coincides with the usual $L_p$-space on the dual group of $\G$.

\medskip

We turn to the second part  on amalgamated free products.
Let $(\A_i,\tau_i)_{i\in I}$ be a family of finite von Neumann algebras equipped with normal faithful normalized traces. Let $\B$ be a common von Neumann subalgebra of $\A_i$ for all $i\in I$, and $\E_i: \A_i\to\B$ the trace preserving conditional expectation. Denote $(\A,\tau)=*_{i, \B}(\A_i,\tau_i)$,  the amalgamated free product of $(\A_i,\tau_i)_{i\in I}$ over $\B$. We will briefly recall the construction to fix notation.

 For any $x\in \A_i$, we denote by $\mathring{x}=x-\E_i x$ and $\mathring \A_i=\{\mathring{x}: x\in \A_i\}$, which yields a natural decomposition $\A_i=\B\oplus \mathring \A_i$. We use the multi-index notation: $\un i=(i_1,\cdots,i_n)\in I^n$. The space
  $$\m W =\B\oplus \bigoplus_{n\geq 1} \bigoplus_{\substack{(i_1,\cdots,i_n)\in I^n \\ i_1\neq i_2\neq\cdots\neq i_n}}\plB {\m A} i n
  =\bigoplus_{n\geq0}\bigoplus_{\substack{(i_1,\cdots,i_n)\in I^n \\ i_1\neq i_2\neq\cdots\neq  i_n}} \m W_{ \un i}$$ 
 is a $*$-algebra by using concatenation and centering with respect to $\B$. The natural projection $\E$ from $\m W$ onto $\B$ is a conditional expectation. $\tau_i\circ\E$ is a trace on $\m W$ that is independent of $i$ and will be denoted by $\tau$. Then $(\A,\tau)$ is the finite von Neumann algebra obtained by the GNS construction from $(\m W,\tau)$. Thus $\m W$ is weak* dense in $\m A$ and dense in $L_p(\A)$ for $p<\infty$. Moreover, $\E$ extends to a trace preserving conditional expectation from $\A$ onto $\B$. As $\E$ restricts to $\E_i$ on $\A_i$, from now on we skip the index $i$ for the conditional expectations.

For $n\geq 0$ we denote by $P_n$ the natural projection 
  $$P_n: \m W \to \m W_n
  =\bigoplus_{\substack{(i_1,\cdots,i_n)\in I^n \\ i_1\neq i_2\cdots\neq  i_n}}\plB {\m A} i n$$ 
 and 
 $P_{\geq n}={\rm Id}- (P_1+\cdots+P_{n-1}).$
Note that $P_n$ extends to a completely bounded projection on $L_p(\A)$ with cb-norm $\leq 2n+1$ for all $1\le p\le\8$ (see \cite{RX06}).

For $k\in I$ let
   \begin{eqnarray*}
   \m L_k = \bigoplus_{\un i: i_1=k} \m W_{\un i}\; \textrm{ and }\; \m R_k = \m L_k^*.
   \end{eqnarray*}
 We denote the associated orthogonal projections on $\m W$  by
 $L_k$ and $R_k$. Given a family $\e=(\e_i)_{i\in \{0\}\cup I}$ of elements in  $\B$, and $x\in \m W$, we define
 \beq\label{FH1}
 {\mathcal H}_\varepsilon(x) =\varepsilon_0\, \E(x)+ \sum_{ k\in I} \varepsilon_k L_{ k}(x)
 \;\text{ and }\;  {\mathcal H}_\varepsilon^{{\rm op}}(x)= \E(x)\varepsilon_0^*+ \sum_{k\in I}R_{k}(x)\varepsilon_k^*.
 \eeq
 
The following is the heart of \cite{MR17}, we state it as a lemma for later reference.

\begin{lemma} \label{MR17}
  Let $1<p<\infty$ and $\e=(\e_i)_{i\in \{0\}\cup I}$ be a family of unitaries in the center $\m Z(\m B)$ of $\B$. Then ${\mathcal H}_\varepsilon$ and  ${\mathcal H}_\varepsilon^{{\rm op}}$ extend to completely bounded maps on $L_p(\A)$. Morover,
 $$
    \|{\mathcal H}_\varepsilon (x)\|_{p }\simeq_p \| x\|_{p}\simeq_p\|{\mathcal H}^{{\rm op}}_\varepsilon (x)\|_{p}\,,\quad x\in L_p(\A).
 $$
   \end{lemma}

%%%%%%%%%%%%%%%%%%%%%%%%%%%%%%%%%%%%%%%%%%%%%%%%%%%%%%%%%%%%%%%%%%%%%%%%
%%%%%%%%%%%%%%%%%%%%%%%%%%%%%%%%%%%%%%%%%%%%%%%%%%%%%%%%%%%%%%%%%%%%%%%%
\section{Norms on modules}\label{Norms on modules}
%%%%%%%%%%%%%%%%%%%%%%%%%%%%%%%%%%%%%%%%%%%%%%%%%%%%%%%%%%%%%%%%%%%%%%%%
%%%%%%%%%%%%%%%%%%%%%%%%%%%%%%%%%%%%%%%%%%%%%%%%%%%%%%%%%%%%%%%%%%%%%%%%

This section is a preparatory part for the next one. Here we will introduce the modular versions of the usual row and column $p$-operator spaces and show the corresponding boundedness results. This is quite tedious  but we have to cope it since they will be the key tools for the proofs of our main results on amalgamated  free products in the next section. Our references for Hilbert C*-modules are \cite{La95,Pa73}.

Let $\B$ be a von Neumann subalgebra of a semifinite  von Neumann algebra $(\M,\tau)$ such that the restriction of $\tau$ to $\B$ is semifinite too.  We need to introduce several norms on tensor products related to $L_p$-modules as in \cite{JPX07}. 

Let $E$ be a right Hilbert $\B$-module with $\B$-inner product $\la \cdot,\,\cdot\ra$ (or $\la \cdot,\,\cdot\ra_E$ if we deal with several modules). $E$ is equipped with the norm induced by its inner product. However, we do not  assume that $E$ is complete.
A typical example is, given an index set $\m I$, the module 
  $$\m C_{ \m I}(\B)=\big\{ (x_\alpha)_{\alpha \in \m I}\subset\B\;:\; \sup_{S\subset \m I \textrm{ finite}} \big\| \sum_{\a\in S} x_\alpha^*x_\alpha \big\|_\B<\infty \big\}$$ 
 with inner product $\la  x, \,y\ra=\lim_S \sum_{\alpha\in S} x_\alpha^*y_\alpha $ which is well defined as a weak* limit over finite subsets of $\m I$ for inclusion. By  \cite{La95,Pa73}, a general module $E$ can be embedded into a self-dual Hilbert module and there exist an index set $\m I$ and a right  $\B$-module map $u=(u_\alpha)_{\alpha\in \m I}:E\to \m C_{\m  I}(\B)$ so that $\la x,\,y\ra=\la  u(x),\, u(y)\ra.$

\smallskip

Given $2\leq p<\infty$, we introduce a norm on the amalgamated tensor product $E\tens_\B L_p(\M)$ as follows:  
  $$\| x\|_{E^c\tens_p L_p(\M)}=\Big\| \Big(\sum_{i,j=1}^n   a_i^*\la x_i,x_j\ra a_j \Big)^{1/2}\Big\|_{L_p(\M)}$$ 
 for  $x=\sum_{i=1}^n x_i\otimes a_i\in E\tens_\B L_p(\M)$. Equipped with this norm, $E\tens_\B L_p(\M)$ is denoted by $E^c\tens_p L_p(\M)$ (the superscript $c$ refers to column). In \cite{JPX07}, this norm is denoted by 
  $\big\| \sum_i | x_i \ra a_i\big\|_p.$ 
 As explained in section 3 there,  via the above concrete embedding of $E$ into  $\m C_{\m I}(\B)$,  the map 
  $$u\tens {\rm Id}_{L_p(\M)}: E^c\tens_\B L_p(\M)\to \m C_{\m
   I}(\B)\tens_\B L_p(\M)\subset L_p\big( B(\ell_2(\m I))\overline
 \tens \M)\big)$$ is then an isometry that allows us to view
 $E^c\tens_p L_p(\M)$ as a subspace of the column
subspace of the last space. This also fully justifies that
   $\|\cdot\|_{E^c\tens_p L_p(\M)}$ is indeed a norm.

Similarly, given $F$ a left Hilbert $\B$-module, we define a norm on $L_p(\M)\tens_\B F$ by
  $$\| y\|_{ L_p(\M)\tens_p F^r}=\Big\| \Big(\sum_{i,j=1}^n b_i\la y_i,\,y_j\ra b_j^* \Big)^{1/2}\Big\|_{L_p(\M)}$$ 
for $y=\sum_{i=1}^n  b_i \otimes y_i$. This norm is denoted by $\big\| \sum_i b_i\la y_i|\,\big\|_p$ in \cite{JPX07}. As above, $L_p(\M)\tens_p F^r$ can be identified with a subspace of the row  subspace of $L_p\big( B(\ell_2(\m J))\overline \tens \M\big)$ for some index set $\m J$.

We can gather the two definitions together and  introduce a norm on $E\tens_\B L_p(\M)\tens_\B F$; the resulting space is denoted by $E^c \tens_p L_p(\M)\tens_p F^r$, it isometrically embeds into  $L_p\big( B(\ell_2(\m J, \m I))\overline \tens \M\big)$. This is independent of the choice of the sets $\m I, \m J$.

\smallskip

We will need extra notions when we assume that $E$ and $F$ are
also bimodules (i.e., $\B\subset \mathcal L_\B(E)$, the algebra of
adjointable right $\B$-modular maps on $E$, and $\B\subset
\,_\B\mathcal L(F)$, the algebra of adjointable left $\B$-modular maps
on $F$). Given $E'$ another right Hilbert $\B$-module, we can consider
$E'\otimes_\B E$ as a right $\B$-module with the internal inner
product; consequently, we have the space $(E'\otimes_\B E)^c\otimes_p
L_p(\M)$. Note that the norm of $(E'\otimes_\B E)^c\otimes_p L_p(\M)$
coincides with that of $E^{\prime c}\otimes_p( E^c\otimes_p L_p(\M))$
when viewing $E^c\otimes_p L_p(\M)$  in the column subspace of
$L_p\big( B(\ell_2(\m I))\overline \tens \M\big)$ as above.  Similar
constructions apply to the row case too.  It is clear that all these
operations are naturally associative.

\begin{rk} 
 We will often use without any reference that the above norms are injective.  We mean that if $E'\subset E$, $F'\subset F$ are submodules and $(\M',\tau)\subset (\M,\tau)$ is a semifinite von Neumann subalgebra, then  $E^{\prime c}\tens_p L_p(\M')\tens_p F^{\prime r}$ isometrically sits  in $E^{ c}\tens_p L_p(\M)\tens_p F^{ r}$.
 \end{rk}
 
 \begin{prop}\label{2cb}
   Assume that $T:E\to E$ is a bounded  right $\B$-modular map. Then
   $$\big\| T\tens {\rm Id}_{L_p(\M)} : E^c\otimes_p L_p(\M)\to E^c\otimes_p L_p(\M) \big\| \leq \|T\|.$$
 If additionally $E$ is equipped with a left $\B$-action and $T$ is $\B$-bimodular, then  for  any right Hilbert $\B$-module $E'$ 
 $$\big\|{\rm Id}_{E'} \tens T\tens {\rm Id}_{L_p(\M)} \big \|_{B\big((E'\tens_\B E)^c\otimes_p L_p(\M)\big)}\leq \|T\|.$$
Similar statements hold for left modules too.
 \end{prop}
   
   \begin{proof}
By \cite[Theorem 2.8]{Pa73},  we have  $\la  T(x),\,T(x)\ra \leq\|T\|^2 \la  x,\,x\ra$ for $x\in E$. Let $x_i\in E$ and $b_i\in \B$ for  $1\le i\le n$. Then
 $$\la  \sum_i T(x_i)b_i,\,\sum_i T(x_i)b_i\ra \leq \|T\|^2 \la  \sum_i x_ib_i,\,\sum_i x_ib_i\ra ,$$ 
so we can reinterpret this inequality in the matrix algebra $\mathbb{M}_n(\B)$ as
 $$0\leq \big(\la  T(x_i),\,T(x_j)\ra\big)_{i,j}\leq\|T\|^2 \big(\la  x_i,\,x_j\ra\big)_{i,j}.$$
The latter  immediately implies 
  $$\Big(\sum_{i,j}a_i^*\la T(x_i), T(x_j)\ra a_j \Big)^{1/2}\leq \|T\|^2\Big(\sum_{i,j}a_i^*\la x_i, x_j\ra a_j \Big)^{1/2}$$
  for any $a_i\in L_p(\M)$. This yields the first assertion.

The second  follows from the same argument by noticing that ${\rm Id}_{E'}\tens T$ is bounded on the right Hilbert $\B$-module $E'\tens_\B E$ with norm $\|T\|$. Indeed, for finite families $x_i\in E$ and $y_i\in E'$ with $1\le i\leq n$
  $$\la \sum_i y_i\tens T(x_i),\,  \sum_i y_i\tens T(x_i)\ra_{E'\tens_\B E}=\sum_{i,j} \la T(x_i), \la  y_i,\,y_j\ra_{E'} T(x_j)\ra_E.$$ 
 As a positive element in $\mathbb{M}_n(\B)$, the matrix $\big(\la y_i,\,y_j\ra_{E'}\big)_{i,j}$ can be written as
  $$\big(\la y_i,\,y_j\ra_{E'}\big)_{i,j}=\sum_k \big(b_{k,i}^*\,b_{k,j}\big)_{i,j}\;\text{ with }\; b_{k,j}\in \B.$$ 
Thus we deduce
  \begin{align*}
  \la \sum_i y_i\tens T(x_i),\,  \sum_i y_i\tens T(x_i)\ra_{E'\tens_\B E}
  &\leq \|T\|^2 \sum_k \la\sum_i b_{k,i} x_i,\,\sum_i b_{k,i} x_i\ra_E\\
  &=\|T\|^2  \la \sum_i y_i\tens x_i,\,  \sum_i y_i\tens x_i\ra_{E'\tens_\B E}.
  \end{align*}
 This is the desired boundedness of ${\rm Id}_{E'}\tens T$ on $E'\tens_\B E$. Thus the proposition is proved.
   \end{proof}
 
 Thanks to the identifications recalled at the beginning of this section, the previous proposition immediately implies the following
 
  \begin{rk}
   If  $T:E\to E$ is a bounded  right $\B$-modular map, then for any left $\B$-module $F$ 
  $$\big\| T\tens {\rm Id}_{L_p(\M)}\tens {\rm Id}_{F} : E^c\otimes_p L_p(\M)\otimes_p F^r \to E^c\otimes_p L_p(\M)\otimes_p F^r\big\| \leq \|T\|.$$
 A similar statement holds for bimodules, namely, ${\rm Id}_{E'} \tens T\tens {\rm Id}_{L_p(\M)}\tens {\rm Id}_{F} $ is bounded on  $(E'\otimes_\B E)^c\otimes_p L_p(\M)\otimes_p F^r$ with norm less than or equal to  $\|T\|$ if additionally $E$ is equipped with a left $\B$-action and $T$ is $\B$-bimodular. 
   \end{rk}

\begin{rk}
  Proposition \ref{2cb} is the modular version of the fact that the row and column $p$-operator spaces are homogeneous (i.e., when $\B=\mathbb C$ and $E$ and $F$ are just Hilbert  spaces).
 \end{rk}

A typical situation to which we will apply the previous results  is the case where $E=F=\A$ with $\A$  a finite von Neumann algebra containing $\B$. The right inner product on $\A$ is $\la  x,\,y\ra=\E(x^*y)$ and the left $\la\la x,\,y\ra\ra=\E(xy^*)$, $\E$ being the trace preserving conditional expectation from $\A$ onto $\B$.  In this case, we have $\A^c\otimes_2 L_2(\B)=L_2(\A)$  isometrically. Thus if $T:\A\to \A$ is right $\B$-modular and bounded for $\la\cdot,\,\cdot\ra$, it automatically extends to a map $\wt T : L_2(\A)\to L_2(\A)$ with $\|\wt T\|_{B(L_2(\A))}\leq
 \|T\|$. Since $\| \la  x,\,x\ra\|_\B=\sup_{b\in L_2(\B), \|b\|_2\leq 1} \|xb\|_2$ for $x\in \A$, we actually have $\|\wt T\|_{B(L_2(\A))}= \|T\|$. We may still write $T$ instead of $\wt T$.

 \smallskip
 
 We state this fact as a lemma for later use.

 \begin{lemma}\label{l2norm}
   Let $(\A,\tau)$ be a finite von Neumann  and $\B\subset\A$ a von Neumann subalgebra with the associated conditional expectation $\E$. Assume that $\A$ is equipped with the right $\B$-module inner product $\la  x,\,y\ra=\E(x^*y)$. Then any bounded right $\B$-modular map $T:\A\to A$  extends to a bounded map on $L_2(\A)$ with  $\| T\|_{B(L_2(\A))}= \|T\|$.
 \end{lemma}

 If $T:\A\to \A$ is  completely positive and leaves $\B$ invariant, then it satisfies the above conditions. If additionally $\E\circ T\le \E$, then $T$ also extends to a completely bounded map on $L_p(\A)$.

 \smallskip
 
We will need the following

 \begin{prop}\label{pcb}
 Let $T: L_p(\M) \to L_p(\M)$ be a completely bounded map. Then for any right Hilbert $\B$-module $E$ and left  Hilbert $\B$-module $F$ we have 
 $$ \big\| {\rm Id}_{E} \tens T\tens {\rm Id}_F \big \|_{{\rm cb}( E^c\otimes_p L_p(\M)\tens_p F^r)} \leq \|T\|_{{\rm cb}}.$$
 \end{prop}
 
 \begin{proof}
 This is a direct consequence of the fact that identifying $E^c\otimes_p L_p(\M)\tens F^r$ with a subspace of $L_p\big( B(\ell_2(\m J, \m I))\overline \tens \M\big)$, then ${\rm Id}_{E} \tens T\tens {\rm Id}_F$ acts like ${\rm Id} \otimes T$. 
 \end{proof}
 
 %%%%%%%%%%%%%%%%%%%%%%%%%%%%%%%%%%%%%%%%%%%%%%%%%%%%%%%%%%%%%%%%%%%%%%%%
%%%%%%%%%%%%%%%%%%%%%%%%%%%%%%%%%%%%%%%%%%%%%%%%%%%%%%%%%%%%%%%%%%%%%%%%
 \section{Multipliers on amalgamated  free products}\label{Multipliers on amalgamated  free products}
%%%%%%%%%%%%%%%%%%%%%%%%%%%%%%%%%%%%%%%%%%%%%%%%%%%%%%%%%%%%%%%%%%%%%%%%
%%%%%%%%%%%%%%%%%%%%%%%%%%%%%%%%%%%%%%%%%%%%%%%%%%%%%%%%%%%%%%%%%%%%%%%%

This section is the core of the article and contains our major
novelty. The principal result is Theorem~\ref{letd} that will imply
Theorem~\ref{main3} by iterations. Throughout the section,
$(\A_i,\tau_i)_{i\in I}$ will denote a family of finite von Neumann
algebras containing $\B$ as a common subalgebra. We will use notation
introduced in section~\ref{Preliminaries} on amalgamated free
products, in particular, $(\A,\tau)=*_{i\in I,\B}(\A_i,\tau_i)$.

 %%%%%%%%%%%%%%%%%%%%%%%%%%%%
 \subsection{An intermediate result}\label{An intermediate result}
%%%%%%%%%%%%%%%%%%%%%%%%%%%%

 Given a family $\pi=(\pi_i)_{i\in I}$ of $*$-representations $\pi_i:\A_i\to \A_i$ such that $\pi_i(b)=b$ for all $b\in \B$ and $\bE\circ \pi_i=\bE$, we introduce a  linear map $T_\pi$ on $\m W$ by
 $T_\pi(b)=b$ for $b\in \B$ and 
  $$T_\pi(a_1\tens\cdots\tens a_n)= \pi_{i_1}(a_1) \tens a_2\tens\cdots\tens a_n$$
 for $n\geq 1$ and $a_1\tens\cdots\tens a_n\in \m W_{\un i}$ where $\un i=(i_1,\cdots,i_n)$ and $i_1\neq\cdots\neq i_n$. Define $T_\pi^{{\rm op}}$ on $\m W$  as $T_\pi^{{\rm op}}(x)=T_\pi(x^*)^*$.  Note that both $T_\pi$ and $T_\pi^{{\rm op}}$ commute with the projections $P_n$.

 We aim to show that  $T_\pi$ extends to a bounded map on $L_p(\A)$ for $1<p<\infty$. To this end, we introduce some paraproducts on $\m W \times \m W$  in the manner of \cite{MR17}: for  $x,y\in \m W$
  \begin{align*}
  x\overset{1,0}\dag y&=xy -  \E_\e [H_\e^{{\rm op}}(xH^{{\rm op}}_\e(y) )],\\
  x\overset{0,1}\dag y&= xy - \E_\e [H_\e(H_\e(x)y )],\\
  x\overset{1,1}\dag y&= xy-x\overset{1,0}\dag y-x\overset{0,1}\dag y,
  \end{align*}
where $H_\e$ and $H^{{\rm op}}_\e$ are the free Hilbert transform defined in \eqref{FH1} and $\E_\e$ is the conditional expectation over all possible choices of symmetric independent signs $\e=(\e_i)$.  Note that when $x$ and $y$ are elementary tensors, then $x\overset{1,0}\dag y$ collects in 
$xy$ the parts that do not end in the same algebra as $y$, that is, all letters in $y$ must have been simplified.  Similarly, $x\overset{0,1}\dag y$ collects in $xy$ the parts that do not start in the same algebra as $x$. 

\medskip

 We will need some elementary free algebraic facts.

\begin{lemma}\label{rel1}
 Let  $g\in \m W_{l}$, $h\in \m W_{n}$ with $l,n\geq 0$. 
 \begin{enumerate}[\rm i)]
 \item If $l>n$, then $T_\pi(gh)= T_\pi(g)h$.
 \item If $n>l$, then $T_\pi^{{\rm op}}(gh)=gT_\pi^{{\rm op}}(h)$.
 \item If $l=n$, then $P_{\geq 2}[T_\pi(gh)]= P_{\geq 2}[T_\pi(g)h]$ and   $P_{\geq 2} [T_\pi^{{\rm op}}(gh)]= P_{\geq 2}[gT_\pi^{{\rm op}}(h)]$.
  \end{enumerate}
 \end{lemma}

 \begin{proof}
   One can assume that $g$ and $h$ are elementary tensors. The first item is then clear as the first letter of $g$ cannot be cancelled if $l>n$. The second is obtained by passing to adjoints. Similarly, if $n=l$, $P_{\geq 2}(gh)$ is a sum of elementary tensors that all start with the first letter of $g$ and end with the last letter of $h$, up to a multiplication by an element of $\B$.
   \end{proof}
 
 The following Cotlar type formula immediately follows from the previous lemma.

 \begin{lemma}\label{cotlargeq2}
 For $g,h\in \m W$ we have 
  \begin{equation}\label{cot2}
   P_{\geq 2}\big[T_\pi(g)T_\pi^{{\rm op}}(h)\big]= P_{\geq2}\big[T_\pi(gT_\pi^{{\rm op}}(h)) + T_\pi^{{\rm op}}(T_\pi(g)h)- T_\pi T_\pi^{{\rm op}}(gh)\big].
   \end{equation}
 \end{lemma}
 
 \begin{proof}
 By linearity, it suffices to show the formula for $g\in \m W_{l}$, $h\in \m W_{n}$. Then it remains to check it by using Lemma \ref{rel1} according to the different cases. We omit the details.
 \end{proof}

 \begin{lemma}\label{rel2}
  Let  $g=g_1\tens\cdots\tens g_l\in \m W_{l}$, $h=h_n\tens\cdots\tens h_1 \in \m W_{n}$ with $l,n\geq 0$, and let $g'=g_2\tens\cdots\tens g_l$ and $h'=h_{n}\tens\cdots\tens h_2$.   \begin{enumerate}[\rm i)]
 \item If $n>l$, then $P_1(gh)=\delta_{n,l+1} \bE(gh')h_1$.
 \item If $l>n$, then $P_1(gh)=\delta_{n+1,l}\,g_1\bE(g'h)$.
 \item If $l=n$, then $P_1 (gh)= P_{1}(g_1\bE(g'h')h_1)$. 
 \end{enumerate}
 \end{lemma}
 
 \begin{proof}
 This proof is easy. Let us verify only i).  If $n>l+1$, then both $P_1(gh)$ and $\E(gh')h_1$ vanish. The  case $n=l+1$ is checked by induction on $l$ and $n$ thanks to the following simplification formula: 
 $$gh= x \otimes \widering{g_nh_l}\otimes y + x\, \E(g_lh_n)y,$$
where  $x=g_1\otimes\cdots\otimes g_{l-1}$ and $y=h_{n-1}\otimes\cdots\otimes h_1$.
It then follows that $P_1(gh)=P_1(x\, \E(g_lh_n)y)$. 
   \end{proof}
 
The following is a complement to the Cotlar formula \eqref{cot2}.

 \begin{lemma}\label{cotlarle2}
 For $g,h\in \m W$ we have
 \begin{eqnarray}\label{cot1}
 P_{1}\big(T_\pi(g)T_\pi^{{\rm op}}(h)\big)
 = T_\pi\big[ P_1\big( g\overset{1,0}\dag T_\pi^{{\rm op}}(h)\big)\big] +T_\pi^{{\rm op}}\big[P_1\big( T_\pi(g)\overset{0,1}\dag h\big)\big] + T_\pi\big[P_1\big( g\overset{1,1}\dag h\big)\big].
 \end{eqnarray} 
 \end{lemma}
 
 \begin{proof}
 We will check the  formula according to the three cases in Lemma \ref{rel2} and use notation  there.  

For i), we can assume $n=l+1$, otherwise all terms are 0. Then 
  $$P_1\big(T_\pi(g)T_\pi ^{{\rm op}}(h)\big)= \E\big(T_\pi(g)h'\big)\,\pi_{j_1}(h_1).$$
On the other hand, for elementary tensors $x\in \m W_{(i_1,\cdots,i_l)}$ and $y\in \m W_{(j_n,\cdots ,j_1)}$, we have $x\overset{1,0}\dag y=0$ as the last letter of $y$ in $xy$ is not cancelled.  To deal with $x\overset{0,1}\dag y$,  note that 
  $$H_\e[P_1(H_\e(x)y)]=\e_{i_1}\e_{j_1}P_1(xy).$$
It thus follows that $P_1(x\overset{0,1}\dag y)=P_1(xy)$ since $\bE(xy')=0$ unless
 $i_1=j_2$ but then $j_1\neq j_2$ and $\E_\e(\e_{i_1}\e_{j_1})=0$. Hence the right hand
 side of \eqref{cot1} is exactly $T_\pi^{{\rm op}}\big[P_1\big(T_\pi(g) h\big)\big]=\E\big(T_\pi(g)h'\big)\pi_{j_1}(h_1)$.

\smallskip 

The case ii) is obtained from i) by passing to adjoints.

\smallskip 

For iii), let $x\in  \m W_{(i_1,\cdots, i_n)}$ and $y\in  \m W_{(j_n,\cdots,j_1)}$ be elementary tensors. Then $P_1(xy)=0$ unless $i_k=j_k$  for all $1\le k\le n$. Hence we can assume that $i_1=j_1$. Then noting  that 
 $$H_\e[P_1(H_\e(x)y)]=\e_{i_1}^2P_1(xy)=P_1(xy),$$ 
we deduce $P_1(x\overset{0,1}\dag y)=0$; by symmetry, $P_1(x\overset{1,0}\dag y)=0$ too. Thus the right hand side of \eqref{cot1} becomes
$$T_\pi\big[P_1\big( g\overset{1,1}\dag h\big)\big]=T_\pi\big[P_1\big( g_1\bE(g'h')h_1\big)\big]=\pi_{i_1}\big[P_1\big( g_1\bE(g'h')h_1\big)\big].$$
However, the left hand side is
 $$P_1\big[\pi_{i_1}(g_1) \bE(g'h')\pi_{i_1}(h_1)\big]=P_1\big[\pi_{i_1}\big( g_1\bE(g'h')h_1\big)\big]=\pi_{i_1}\big[P_1\big( g_1\bE(g'h')h_1\big)\big]$$ 
for $\pi_{i_1}$ is a $*$-representation, leaves the elements of $\B$ invariant and $P_1\pi_{i_1}=\pi_{i_1}P_1$.  \eqref{cot1} is thus proved in the case $l=n$ too.
 \end{proof}

The following is an intermediate result to Theorem~\ref{letd}, it will be the key for the reduction formula in Theorem~\ref{redform} below.

 \begin{thm}\label{rep}
  The map $T_\pi$ extends to a completely bounded map on $L_p(\A)$ for all $2\leq p<\8$ with cb-norm majorized by a constant depending only on $p$.
 \end{thm}

 \begin{proof}
 As the $\pi_i$'s are trace preserving and leave $\B$ invariant,
 $T_\pi$ is an isometry on $\m W$ for the $L_2(\A)$-norm, thus it
 extends to an isometry on $L_2(\A)$.  We need only
 to prove the boundedness of $T_\pi$ since the complete boundedness
 will then be automatic thanks to the usual trick of replacing $\B$ by
 the matrix algebra $\mathbb{M}_n(\B)$.

As in \cite{MR17}, we show that the $L_p$-boundedness of $T_\pi$  implies its $L_{2p}$-boundedness for $2\leq p<\infty$. Starting with $p=2$, using iteration and interpolation, we will deduce the assertion for the full range $2<p<\infty$. In the following, we will denote by $\gamma_p$ the norm of $T_\pi$ and $T_\pi^{{\rm op}}$ on $\m W$ equipped with the  $L_p$-norm. 

For  $x\in \m W$ we write
 \begin{align*}
 T_\pi(x)T_\pi(x)^*&=T_\pi(x)T_\pi^{{\rm op}}(x^*)\\
 &=P_{\geq 2}[T_\pi(x)T_\pi^{{\rm op}}(x^*)]+P_1[T_\pi(x)T_\pi^{{\rm op}}(x^*)]+ \E [T_\pi(x)T_\pi^{{\rm op}}(x^*)].
 \end{align*}
By \eqref{cot2} and the fact that $P_{\geq 2}$ has norm less than 5 on $L_p(\A)$, we have
 $$\big\| P_{\geq 2}[T_\pi(x)T_\pi^{{\rm op}}(x^*)]\big\|_p\leq 5\,\big[2\gamma_p  \|x\|_{2p}\|T_\pi(x)\|_{2p}+\gamma_p^2\|x\|_{2p}^2\big].$$ 
On the other hand, by \cite[Proposition 3.14]{MR17} or Lemma~\ref{MR17}, the paraproducts $\overset{i,j}\dag$ are bounded from $L_{2p}(\A)\times L_{2p}(\A)$ to $L_p(\A)$ with norm less than $\eta_p$. Thus using \eqref{cot1} and the fact that $P_{1}$ has norm less than 3 on $L_p(\A)$, we get
  $$\big\| P_{1}[T_\pi(x)T_\pi^{{\rm op}}(x^*)]\big\|_p\leq 3\eta_{p}\gamma_p \big[2\|x\|_{2p}\|T_\pi(x)\|_{2p} + \|x\|_{2p}^2\big].$$
Clearly, 
 $$\big\|\bE [T_\pi(x)T_\pi^{op}(x^*)]\big\|_p=\|\bE(xx^*)\|_p\le \|x\|_p^2.$$ 
So combining all the estimates yields
 $$\|T_\pi(x)\|_{2p}^2 \leq \gamma_p \big(10+6\eta_{p})\|x\|_{2p}\|T_\pi(x)\|_{2p}+ (1+3\eta_{p}\gamma_p+5\gamma_p^2) \|x\|_p^2.$$
 It thus follows that
  $$\|T_\pi(x)\|_{2p}\leq C\gamma_p \eta_{p}\|x\|_{2p}$$ 
for some absolute constant $C$, whence $\g_{2p}\le C\gamma_p \eta_{p}$. This finishes the proof.
 \end{proof}

%%%%%%%%%%%%%%%%%%%%%%%%%%%%
  \subsection{A length reduction formula}
%%%%%%%%%%%%%%%%%%%%%%%%%%%%

 We show here how to recursively estimate the $L_p$-norm of an element in 
 $\m W$ in the spirit of \cite{JPX07}. The space $\m W$ is naturally a right Hilbert $\B$-module with inner product $\la x,\,y\ra  =\E(x^*y)$ and also a left Hilbert $\B$-module with $\la\la x,\,y\ra\ra =\E(xy^*)$. The same holds for  $\m W_1$ and $\mathring {\m W}$ as submodules. A typical element in $\m W$ can be written as a finite sum 
 \beq\label{dec}
 x= x_0 + x_1 + \sum_{i,\alpha} a_i(\alpha)\otimes b_i(\alpha){\mathop=^{\rm def}} x_0 + x_1 +z,
 \eeq
 where $x_0\in \B$, $x_1\in \m W_1$ and $a_i(\alpha)\in \mathring{\m A_i}$ and 
 $b_i(\alpha)\in \mathring {\m W}$ with $L_i(b_i(\alpha))=0$.

The following result extends the main result of \cite{JPX07} on homogeneous polynomials to any polynomials, which is the key tool for the argument in the next subsection.

 \begin{thm}\label{redform}
   With the notation above, for $2\le p<\infty$,  we have $($considering $z\in
   \mathcal W_1\otimes \mathring{\m W})$
  \beq\label{rec}
  \|x\|_p 
   \approx_p  \|x_0\|_p + \|x_1\|_p + \| z\|_{\m W_1^c\otimes_p L_p(\m A)}  +\|z\|_{L_p(\m A)\otimes_p \mathring{\m W}^r}.
 \eeq
 \end{thm}
 \begin{proof}
 By the boundedness of the projections $P_0$ and $P_1$, it suffices to prove the estimate
 for $z$. To this end, we  consider two copies of $\A$, and put a superscript 
 to distinguish them.  We use associativity of the free product to write
 $$ \A^{(1)} *_\B  \A^{(2)}=\big (*_{i\in I, \,\B} \A_i^{(1)}\big)*_\B \big(*_{i\in I,\,\B} \A_i^{(2)}\big)=*_{i\in I,\,\B} \big(\m A_i^{(1)}*_{\B}\m A_i^{(2)}\big){\mathop=^{\rm def}}*_{i\in I,\,\B}\wt {\m A}_i.$$
 Since the traces are compatible,
 $\|x\|_{L_p(\m A)}= \|x^{(1)}\|_{L_p(\m A^{(1)}*_\B \m A^{(2)})}$ for every $x\in L_p(\A)$.

 For each $i$, we define the swap map $\pi_i$ on $\wt  {\m A}_i$ by
  \begin{align*}
  &a_1^{(1)}\tens a_2^{(2)}\tens a_3^{(1)}\tens\cdots\mapsto a_1^{(2)}\tens a_2^{(1)}\tens a_3^{(2)}\tens\cdots\\
  &a_1^{(2)}\tens a_2^{(1)}\tens a_3^{(2)}\tens\cdots\mapsto a_1^{(1)}\tens a_2^{(2)}\tens a_3^{(1)}\tens\cdots
  \end{align*}
for $\B$-centered elements $a_k^{(j)}$ and by $\pi_i(b)=b$ for $b\in \B$. It is clear that $\pi_i$ is a $*$-representation, $\bE$-preserving and leaves the elements of $\B$ invariant. Thus, noticing that $\pi_i^2={\rm Id}_{\wt {\m A}_i}$, we use Theorem \ref{rep} to get
  $$\Big\| \sum_{i,\alpha} a_i(\alpha)\otimes b_i(\alpha)\Big\|_{L_p(\m A)}\simeq_p 
  \Big\| \sum_{i,\alpha} a_i(\alpha)^{(2)}\otimes b_i(\alpha)^{(1)}\Big\|_{L_p(\m A^{(1)}*_\B\m A^{(2)})}.$$ 
 The element
 $\sum_{i,\alpha} a_i(\alpha)^{(2)}\otimes b_i(\alpha)^{(1)}$ is
 homogeneous of degree 2 with respect to the length of $\m A^{(1)}*_\B\m
 A^{(2)}$.  Thus by \cite[Theorem~B]{JPX07} 
  \begin{align*}
  \Big\|\sum_{i,\alpha} a_i(\alpha)^{(2)}\otimes b_i(\alpha)^{(1)}\Big\|_{L_p(\m A^{(1)}*_\B\m A^{(2)})}
  \simeq
  \Big\| \sum_{i,\alpha} |a_i(\alpha)^{(2)}\ra b_i^{(1)}(\alpha)\Big\|_{p}  
  +\Big\| \sum_{i,\alpha} a_i(\alpha)^{(2)}  \la b_i^{(1)}(\alpha)|\Big\|_{p}\,.
  \end{align*}
Since taking copies clearly does not change norms, we have
  \begin{align*}
  \Big\| \sum_{i,\alpha} |a_i(\alpha)^{(2)}\ra b_i^{(1)}(\alpha)\Big\|_{p}  
  &=\Big\| \sum_{i,\alpha} a_i(\alpha)^{(2)}\otimes b_i^{(1)}(\alpha)\Big\|_{(\m W_1^{(2)})^c\ot_pL_p(\m A^{(1)}*_\B\m A^{(2)})}\\
  &= \Big\| \sum_{i,\alpha} a_i(\alpha)\otimes b_i(\alpha)\Big\|_{\m W_1^c\otimes_p L_p(\m A)}\,.\\
    \end{align*}
  Similarly, 
  $$\Big\| \sum_{i,\alpha} a_i(\alpha)^{(2)} \la b_i^{(1)}(\alpha)|\Big\|_{p}
  = \Big\| \sum_{i,\alpha} a_i(\alpha)\tens  b_i(\alpha)\Big\|_{L_p(\m A)\otimes_p  {\mathring {\m W}}^r}.$$
This concludes the proof of the theorem.
 \end{proof}

 \begin{rk}\label{explicit}
 We could also have used \cite[Theorem C]{JPX07} to make some terms  more
 explicit. Namely,
 $$ \|x_1\|_p \simeq \big\|\big(\bE(x_1^*x_1)\big)^{1/2}\big\|_p + \Big(\sum_i \|L_i(x_1)\|_p^p\Big)^{1/p} + \big\| \big(\bE(x_1x_1^*)\big)^{1/2}\big\|_p,$$
 and 
 \begin{align*}
 \|z\|_{L_p(\m A)\otimes_p \mathring {\m W}^r}
 &\simeq_p  \big\|z \big\|_ {\m W_1^c\otimes\mathring {\m W}^r}+ \big\|\big(\bE(zz^*)\big)^{1/2}\big\|_p \\  
 & \quad+
  \Big(\sum_{i} \big\| \big(\sum_{\alpha,\beta} a_i(\alpha)\bE( b_i(\alpha)b_i(\beta)^*)a_i(\beta)\big)^{1/2} \big\|_p^p\Big)^{1/p}.
 \end{align*}
 Here $\m W_1^c\otimes\mathring {\m W}^r$ has to be understood as
 $\m W_1^c\otimes_p L_p(\B)\otimes_p\mathring {\m W}^r$.
 \end{rk}

 \begin{rk}\label{jpxb}
  It is now rather easy to get an analogue of \cite[Theorem
    C]{JPX07}. The norm of the last term $\|z\|_{\m W_1^c\otimes_p
    L_p(\m A)}$ corresponds to that of an element in $S_p\otimes_p
  L_p(\m A)$ and we can formally iterate the argument. Thus, we can write the norm of $x\in P_{\geq
    k}\big(L_p(\A)\big)$ as a sum of $2k+1$ norms. For simplicity
  assume $x\in P_{\geq k}(\m W)$, they are given by
  \begin{align*}
    & \| x\|_{\mathcal W_l^c\otimes\m W^r}, \quad 0\leq l\leq k, \\
    &  \| x\|_{\mathcal W_l^c\otimes_p (\oplus_p \mathring{L_p(\m A_i))}\otimes_p \m W^r}, \quad 0\leq l\leq k-2 \\
  &  \| x\|_{\mathcal W_{k-1}^c\otimes_p L_p(\m A)}.  
    \end{align*}
The last one being recursive. We leave the details
  to the interested reader.
 \end{rk}
     
%%%%%%%%%%%%%%%%%%%%%%%%%%%%
 \subsection{Maps of the $d$-th letters and the proof of Theorem~\ref{main3}}
%%%%%%%%%%%%%%%%%%%%%%%%%%%%

This subsection contains our principal result that is the key step of the proof of Theorem~\ref{main3}. Fix $1<p<\8$. Given a family of maps $T_i: L_p(\A_i)\to L_p(\A_i)$ we will define an associated map of the $d$-th letters of reduced words in $\m W$ for $d\ge1$. The minimal assumption required for the $T_i$'s is the following 
 \begin{itemize}
 \item[{\bf (H$_1$)}]  $T_i$ is  $\B$-bimodular and  $T_{i}\big(\widering{L_q(\A_i)}\big)\subset \widering{L_q(\A_i)}$ for $q=2$ and $q=p$.
  \item[{\bf (H$_2$)}] $T_i: L_q(\A_i)\to L_q(\A_i)$ is completely bounded and
 $${\rm cb}_q=\sup_{i\in I} \|T_i\|_{{\rm cb}(L_q(\m A_i))}<\infty \;\;\text{ for }q=2 \text{ and }q=p.$$
 \end{itemize}
Note that the cb-norm of $T_i$ on $L_2(\m A_i)$ coincides with its usual norm for $L_2(\m A_i)$ is a homogeneous operator space. On the other hand, it is obvious that
 $${\rm cb}_q=\big\|\oplus_i T_i: L_q(\oplus_i\A_i)\to  L_q(\oplus_i\A_i)\big\|_{\rm cb}\,.$$
Now we define a linear map $T^{(d)}$ on $\m W$ by $T^{(d)}(b)=b$ for $b\in \B$ and 
  $$T^{(d)}(a_1\tens\cdots\tens a_n)
  = \left\{\begin{array}{lcl}
 a_1 \tens\cdots\tens a_{d-1}\tens T_{i_d}(a_d)\tens a_{d+1}\tens\cdots\tens a_n& \textrm{if}& d\le n,\\
 a_1 \tens \cdots \tens a_n& \textrm{if}& d> n\end{array}\right. $$
 for $n\geq 1$ and $a_1\tens\cdots\tens a_n\in \m W_{\un i}$ with $\un i=(i_1,\cdots,i_n)$ and $i_1\neq\cdots\neq i_n$.  Note that the range
   of $T^{(d)}$ is not inside $\m W$ but clearly in $L_p(\m A)$.

\smallskip
 
\begin{thm}\label{letd}
 Under the hypotheses {\bf (H$_1$)} and {\bf (H$_2$)},  $T^{(d)}$ extends to a completely bounded map on $L_p(\A)$ with
 $$\|T^{(d)}\|_{{\rm cb}(L_p(\A))}\lesssim_{p,d}   {\rm cb}_2+{\rm cb}_p.$$
\end{thm}

\begin{proof}

We start the proof by a crucial observation related to  Lemma \ref{l2norm}. We view $\A_i$ as a right Hilbert $\B$-module with the inner product $\la x,\,y\ra=\E(x^*y)$.  We claim that $\E(T_i(x)^*T_i(y))$ belongs to $\B$ for any $x, y\in\A$.  By polarization, we can assume $x=y$. Then for $b\in L_2(\A_i)$, by the modularity of $T_i$, 
 $$\tau[b^*\E(T_i(x)^*T_i(x))b]=\tau[\E(T_i(xb)^*T_i(xb))]=\|T_i(xb)\|_2\le\|T_i\|\,\|xb\|_2\,.$$
This implies the claim, as well as $\E(T_i(x)^*T_i(x))\le \|T_i\|\la
x,\, x\ra$. Thus $T_i(\A_i)$ is a $\B$-bimodule.  A similar statement holds when $\A_i$ is viewed as a left
$\B$-module with the inner product $\la\la x,\,y\ra\ra=\E(xy^*)$.

We have stated the results in section~\ref{Norms on modules} for maps $T:E\to E$ but they clearly remain true if $T:E\to  E'$.  We can thus apply them to the restriction of $T_i$ to $\A_i$ with  range $T_i(\A_i)$. The same holds for direct sums.
  
An immediate consequence is the boundedness of $T^{(d)}$ on $\m W$ equipped with the $L_2$-norm, this follows from the orthogonality of the $\m W_n$'s and Proposition \ref{2cb} combined with the above observation. So the theorem holds for $p=2$. On the other hand, by duality, we need only to consider the case $2<p<\8$ that will be assumed in the remainder of the proof.

Before going into the core of the proof, we point out that the
  result of Theorem \ref{redform} can be applied to $T^{(d)}(x)$ when
  $x\in \m W$ using an obvious approximation argument. One just need
  to adapt correctly the modules, replacing $\m W_1^c$ by $ T^{(1)}(\m
  W_1)^c$ if $d=1$ and $\mathring{\m W}^r$ by ${
    T^{(d-1)}(\mathring{\m W})}^r$ if $d\geq1$.

We continue the proof by induction on $d$.  The main part is the initial step: $d=1$.  By the  usual argument of tensoring with the matrix algebras $\mathbb{M}_n$, it suffices to prove the boundedness of $T^{(1)}$. We will apply Theorem \ref{redform}. Let $x= x_0 + x_1 +z\in \m W$ as \eqref{dec}.

To deal with $\|T^{(1)}(x_1)\|_p$, we use the Khintchine inequality from Remark \ref{explicit}. We have that $T^{(1)}$ is bounded on $L_p(\oplus_i\A_i)$ with norm majorized by ${\rm cb}_p$, that is,
 $$\Big(\sum_i \|L_i(T^{(1)}(x_1))\|_p^p\Big)^{1/p}=\Big(\sum_i \|T_i(L_i(x_1))\|_p^p\Big)^{1/p}\le {\rm cb}_p\Big(\sum_i \|x_1\|_p^p\Big)^{1/p}\,.$$ 
On the other hand, thanks to the previous observation,  $\oplus_i T_i$ can be viewed as a  bounded modular map on $\m W_1$ with respect to  both inner products $\la \cdot,\,\cdot\ra$ and $\la\la \cdot,\,\cdot\ra\ra$ with norm bounded by ${\rm cb}_2$. Thus by Proposition \ref{2cb},
   $$\big\|\big(\bE[T^{(1)}(x_1)^*T^{(1)}(x_1)]\big)^{1/2}\big\|_p=\big\|\big(\bE[([\oplus_iT_i](x_1))^*[\oplus_i T_i](x_1)]\big)^{1/2}\big\|_p\le {\rm cb}_2\big\|\big(\bE(x_1^*x_1)\big)^{1/2}\big\|_p$$
 and similarly for the second inner product. Hence,
  $$\|T^{(1)}(x_1)\|_p\lesssim( {\rm cb}_2+{\rm cb}_p) \|x_1\|_p.$$
For the remaining part $z$, note that
 $$T^{(1)}(z)=\sum_{i,\alpha} [\oplus_j T_j](a_i(\alpha))\otimes b_i(\alpha).$$
Thus by the observation that $\oplus_i T_i: \m W_1\to T^{(1)}(\m W_1)$ is $\B$-bimodular and Proposition \ref{2cb}, we again have
 $$\Big\| \sum_{i,\alpha} T_i(a_i(\alpha))\otimes b_i(\alpha)\Big\|_{{T^{(1)}(\m W_1)^c}\otimes_p L_p(\m A)}\le  {\rm cb}_2
  \Big\| \sum_{i,\alpha} a_i(\alpha)\otimes b_i(\alpha)\Big\|_{{\m W_1^c}\otimes_p L_p(\m A)}.$$
For the other norm of $z$, we use Proposition \ref{pcb} to obtain 
 $$\Big\| \sum_{i,\alpha} T_i(a_i(\alpha))\otimes b_i(\alpha)\Big\|_{L_p(\m A)\otimes_p \mathring{\m W}^r}\le  {\rm cb}_p
  \Big\| \sum_{i,\alpha} a_i(\alpha)\otimes b_i(\alpha)\Big\|_{L_p(\m A)\otimes_p \mathring{\m W}^r}.$$
Thus $T^{(1)}$ extends to a (completely) bounded map on $L_p(\A)$.

\medskip

Now assume that $d\geq2$ and $T^{(d-1)}$ is completely bounded on $\m W$ for the $L_p$-norm. For $x=x_0+x_1+z$ as above, we have
 $$T^{(d)}(x_0+x_1)=x_0+x_1\;\text{ and }\;T^{(d)}(z)= \sum_{i,\alpha} a_i(\alpha)\otimes T^{(d-1)} (b_i(\alpha)).$$
Using the boundedness of $T^{(d-1)}$ on $L_2(\A)$ and Proposition \ref{2cb}, we have
 $$\Big\| T^{(d)} \Big(\sum_{i,\alpha} a_i(\alpha)\otimes b_i(\alpha)\Big)\Big\|_{L_p(\A)\otimes_p {T^{(d-1)}(\mathring{\m W})^r}}\le {\rm cb}_2
 \Big\|\sum_{i,\alpha} a_i(\alpha)\otimes b_i(\alpha)\Big\|_{L_p(\A)\otimes_p \mathring{\m W}^r}.$$
Similarly, the  complete boundedness of $T^{(d-1)}$ on $L_p(\A)$ and Proposition \ref{pcb} imply
 $$\Big\| T^{(d)} \Big(\sum_{i,\alpha} a_i(\alpha)\otimes b_i(\alpha)\Big)\Big\|_{\m W_1^c\otimes_p L_p(\m A)}\lesssim_{p, d}   {\rm cb}_p
 \Big\|\sum_{i,\alpha} a_i(\alpha)\otimes b_i(\alpha)\Big\|_{\m W_1^c\otimes_p L_p(\m A)}.$$
This concludes the induction, and the proof of the theorem too.
\end{proof}

 \medskip
 
 Theorem~\ref{main3} immediately follows from Theorem~\ref{letd}.
 
\begin{proof} [Proof of Theorem \ref{main3}] 

For $1\le k\le d$ let $T_k^{(k)}$ be the map $T^{(k)}$ in Theorem \ref{letd} associated to the family $(T_{k,i})_{i\in I}$. Then
 $$T^{Ld}=T_1^{(1)} T_2^{(2)}\cdots T_d^{(d)}.$$
This yields the assertion.\end{proof}

We extend Lemma~\ref{MR17} to the Hilbert transform of the $d$-th letters in the spirit of  Theorem \ref{letd} and Theorem~\ref{main3}.   Let $\e=(\e_i)_{i\in I}$ be a family of elements in the unit ball of $\m Z(\m B)$. Let $T_i$ be the left multiplication map  on  $L_q(\A_i)$ by $\e_i$. Clearly, the $T_i$'s satisfy the hypotheses {\bf (H$_1$)} and {\bf (H$_2$)}. Denote the corresponding $T^{(d)}$ by $\mathcal{H}_\e^{(d)}$. If $d=1$, this coincides with the free Hilbert transform in \eqref{FH1}.  More generally, given $d$ let $\e=(\e_{j,i})_{1\le j\le d,i\in I}$ be a family  in the unit ball of $\m Z(\m B)$. The corresponding map $T^{Ld}$ as in Theorem~\ref{main3} is denoted by $\mathcal{H}_\e^{Ld}$. 

\medskip

The following is a particular case of Theorem \ref{main3}, it extends \cite[Theorem 4.7]{MR17} to the amalgamated free product case.
 
 \begin{corollary}\label{letHd}
  Both $\mathcal{H}_\e^{(d)}$ and $\mathcal{H}_\e^{Ld}$ extend to completely bounded maps on  $L_p(\m A)$ for all $1<p<\infty$ with cb-norms controlled by constants depending only on $p$ and $d$.
   \end{corollary}
 
 We conclude this section with the boundedness of some paraproducts that generalize those introduced at the beginning of subsection~\ref{An intermediate result}.  These paraproducts are of independent interest in free analysis. 
 
 Let $\e=(\e_{j,i})_{j\ge1,i\in I}$ be an independent family of symmetric random variables with values $\pm1$. Let $\mathcal{H}_\e^{Ld}$ be the  map associated to $(\e_{j,i})_{1\le j\le d,i\in I}$ as in the previous corollary, and let $\mathcal{H}_\e^{Ld, {\rm op}}(x)=[\mathcal{H}_\e^{Ld}(x^*)]^*$. We use the convention that $\mathcal{H}_\e^{L0}={\rm Id}$.  For $j, k\ge0$ and $x, y\in\m W$, define
 $$
  x\overset {j,k}\top  y=\E_\e\,\E_{\e'} \,\mathcal{H}_\e^{Lj}\,\mathcal{H}_{\e'}^{Lk, {\rm op}}[\mathcal{H}_\e^{Lj}(x){\mathcal H}_{\e'}^{Lk, {\rm op}}(y)],
 $$
 where $\E_\e$ denotes the underlying expectation and $\e'$ is an independent copy of $\e$. This paraproduct is easily understood for elementary tensors $x$ and $y$: $x\overset {j,k}\top  y$ then collects all those terms in the development of $xy$ into elementary tensors whose first $j$ letters come from the same algebras of the  first $j$ letters of $x$, and whose last $k$ letters   from the same algebras of the  last $k$ letters of $y$. 
 
The previous corollary implies the following

\begin{corollary}\label{freepara}
 The paraproduct $\overset {j,k}\top$ extends to a bounded bilinear map from $L_{2p}(\A)\times L_{2p}(\A)$ to $L_{p}(\A)$ for all $1<p<\8$ with norm majorized by a constant depending only on $p, j$ and $k$.
   \end{corollary}

\begin{rk}
 The reader familiar with Haagerup noncommutative $L_p$-spaces can extend, with necessary modifications,  all results of this section to the type III case, that is, to amalgamated free products of von Neumann algebras equipped with faithful normal states instead of traces.
\end{rk}

%%%%%%%%%%%%%%%%%%%%%%%%%%%%%%%%%%%%%%%%%%%%%%%%%%%%%%%%%%%%%%%%%%%%%%%%
%%%%%%%%%%%%%%%%%%%%%%%%%%%%%%%%%%%%%%%%%%%%%%%%%%%%%%%%%%%%%%%%%%%%%%%%
\section{Multipliers on free products of groups}\label{Multipliers on free products of groups}
%%%%%%%%%%%%%%%%%%%%%%%%%%%%%%%%%%%%%%%%%%%%%%%%%%%%%%%%%%%%%%%%%%%%%%%%
%%%%%%%%%%%%%%%%%%%%%%%%%%%%%%%%%%%%%%%%%%%%%%%%%%%%%%%%%%%%%%%%%%%%%%%%

In this section we will first prove Theorems~\ref{main1} and  \ref{main2}, then consider Fourier multipliers on free products of general discrete groups.  Recall that $\wh\G$ denotes the group von Neumann algebra of a discrete group $\G$ generated by the left regular representation $\l$.

%%%%%%%%%%%%%%%%%%%%%%%%%%%%
 \subsection{Proofs of Theorems~\ref{main1} and \ref{main2}}
%%%%%%%%%%%%%%%%%%%%%%%%%%%%

We start by the results on the free group $\mathbb F_\infty$. 

\begin{proof}[Proof of Theorem~\ref{main2}]  We will apply Theorem \ref{main3} to the special case where $\A_i=\wh\ent=L_\infty(\T)$ for all $i\in \mathbb N$.  Fix a family $z=(z_{j,i})_{1\le j\le d,\, 1\le i<\8}$ of complex numbers with modulus 1. Define  $T_{j,i}$ to be the measure preserving $*$-representation on $\A_i$ given by $T_{j,i}(\l(n))=z_{j,i}^n\l(n)$ for any  $n\in\ent$, or equivalently in terms of the generator $\zeta\in L_\infty(\T)$, $T_{j,i}(\zeta^n)=z_{j,i}^n \zeta^n$. $T_{j,i}$ extends to a complete isometry on $L_p(\T)$ for $1\leq p<\infty$.  The corresponding map $T^{Ld}$ is exactly the map $\alpha_z^{Ld}$ in Theorem~\ref{main2}.  Thus Theorem~\ref{main3} implies that $\alpha_z^{Ld}$ is completely bounded on $L_p(\wh\F_\8)$ for $1<p<\8$, whence Theorem~\ref{main2}. \end{proof}

\begin{proof}[Proof of Theorem~\ref{main1}]
We will use $\alpha_z^{Ld}$ in the previous proof for the special case where $z_{j,i}=z_j$ for all $i$, and  write $\alpha_z=\alpha_z^{Ld}$ for  $z\in \T^d$.  Thus $\a$ is a uniformly completely bounded action of $\T^d$ on $L_p(\wh\F_\8)$. We then easily deduce Theorem~\ref{main1} by the standard transference argument as presented in \cite{CW76}. Let us give the details.

Let $m$ be a H\"ormander-Mikhlin multiplier on ${\mathbb Z}^d$, that
is, $m$ satisfies \eqref{HM}. Then the associated Fourier multiplier
$T_m$ on $L_p(\T^d)$ is completely bounded with cb-norm majorized by
$C_{p,d}\|m\|_{{\rm HM}}$. This follows from \cite{Bou86} for $d=1$
and \cite{Mc94,Zi89} for $d\ge2$ since the Schatten $p$-class $S_p$ is
a UMD space. Note that valid for general UMD spaces, the results in
these papers require more regularity on $m$ than the condition
\eqref{HM}, that is, the partial discrete derivations should run to
all orders up to $d$ instead of $[\frac d2]+1$ in \eqref{HM}. However, using the arguments of
\cite[section~4.1]{XXY18}, we can show that when the UMD space in
consideration is a noncommutative $L_p(\M)$, we can go down again to
the classical order $[\frac d2]+1$. Hence
 $$\big\|T_m\ot {\rm Id}_{S_p}: L_p(\T^d; S_p)\to L_p(\T^d;
S_p)\big\|_{\rm cb}\les_{p, d}\|m\|_{{\rm HM}}\,.$$ The Schatten
$p$-class $S_p$ here can be replaced by $L_p(\M)$ for any QWEP $\M$,
in particular, by $L_p(\wh {\F}_\8)$. Thus
 $$\big\|T_m\ot {\rm Id}_{L_p(\wh {\F}_\8)}: L_p(\T^d; L_p(\wh {\F}_\8))\to L_p(\T^d; L_p(\wh {\F}_\8))\big\|_{\rm cb}\les_{p, d}\|m\|_{{\rm HM}}\,.$$
Now given $x\in L_p(\wh {\F}_\8)$ define $f\in L_p(\T^d; L_p(\wh {\F}_\8))$ by $f(z)= \alpha_z(x)$ for $z\in\T^d$. Then by Theorem~\ref{main2}
  $$\|f(z)\|_{L_p(\wh {\F}_\8)}\simeq_{p, d} \|x\|_{L_p(\wh {\F}_\8)},\quad z\in\T^d.$$
Clearly, we have the intertwining identity:
 $$[T_m\ot {\rm Id}_{L_p(\wh {\F}_\8)}](f)(z)=\alpha_z(M_m(x)),\quad z\in\T^d.$$
Thus we deduce
  \begin{align*}
  \big\|M_m(x)\big\|^p_{L_p(\wh {\F}_\8)}
  &\les_{p, d}\int_{\T^d}\big\|\alpha_z(M_m(x))\big\|^p_{L_p(\wh {\F}_\8)}dz\\
 &=\big\|[T_m\ot {\rm Id}_{L_p(\wh {\F}_\8)}](f)\big\|^p_{L_p(\T^d; L_p(\wh {\F}_\8))}\\
 &\les_{p, d}\|m\|^p_{{\rm HM}}\,\big\|f\big\|^p_{L_p(\T^d; L_p(\wh {\F}_\8))}\\
 &\les_{p, d}\|m\|^p_{{\rm HM}}\,\big\|x\big\|_{L_p(\wh {\F}_\8)}\,.
 \end{align*}
Therefore, $M_m$ is bounded on $L_p(\wh {\F}_\8)$ with norm controlled by $C_{p,d}\|m\|_{{\rm HM}}$. The complete boundedness follows from the usual argument of tensoring with $S_p$.
 \end{proof}

We end this subsection with some examples of Fourier multipliers on the free group. The free Hilbert transform of \cite{MR17} is a typical example of Fourier multipliers studied in this article. Theorem~\ref{main1} allows us to exhibit plenty of examples of Fourier multipliers on the free group. We give here just some typical ones.

\begin{ex}
Let $\A_i=L_\infty(\T)$ for all $i\in \mathbb N$ and $z=(z_{j,i})_{1\le j\le d,i\in I}$ be a family  of complex numbers of modulus 1. Then the corresponding transforms ${\mathcal H}_z^{(d)}$ and ${\mathcal H}_z^{Ld}$  in Corollary~\ref{letHd}  are completely bounded on $L_p(\wh\F_\8)$ for all $1<p<\8$. 
\end{ex}

\begin{ex}
We give two more examples of similar nature: For $1\le j\le d$ let
$T_{j,i}$ be the map on $\A_i=L_\infty(\T)$ defined by
$T_{j,i}(\zeta^n)=z_{j,i}^{{\rm sgn}(n)}\zeta^n$ for $\zeta\in\T$ and
$n\in\ent$. Clearly, $T_{j,i}$ is completely bounded on $L_p(\T)$
 for all $1< p<\8$. Let
$\wt{\mathcal H}_z^{Ld}$ be the corresponding map $T^{Ld}$ in
Theorem~\ref{main3} and $\wt{\mathcal H}_z^{(d)}$ the map $T^{(d)}$ in
Theorem~\ref{letd} associated to $(T_{d,i})_{i\in\nat}$. Again,
$\wt{\mathcal H}_z^{Ld}$ and $\wt{\mathcal H}_z^{(d)}$ are completely
bounded on $L_p(\wh\F_\8)$ for all $1<p<\8$.
\end{ex}

\begin{ex}
The Riesz transforms $R_j$, $1\le j\le d$, on $L_p(\T^d)$ are the Fourier multipliers of symbols
 $$m_j(k)=\frac{k_j}{\|k\|}\;\text{ for }k=(k_1,\cdots, k_d)\in\ent^d, \;\; \|k\|=\sqrt{k_1^2+\cdots+k_d^2}.$$
It is classical that $R_j$  is completely bounded on  $L_p(\T^d)$ for $1<p<\8$. The corresponding multipliers $M_{m_j}$ in Theorem \ref{main1} are denoted by $R_j^{Ld}$ and may be called  the free Riesz transforms of the first $d$ letters on $\F_\8$.  $R_j^{L1}$ is just the free Hilbert transform of \cite{MR17}.  $R_j^{Ld}$, $1\le j\le d$, are completely bounded on $L_p(\wh\F_\8)$ for $1<p<\8$.
\end{ex}

\begin{ex}
Our final example is given by the classical Littlewood-Paley multiplier. It is well known that the following Littlewood-Paley multiplier
 $$m=\sum_{i=0}^\8\e_i\mathds {1}_{\{k\in\ent^d: 2^i-1\le |k|<2^{i+1}\}}\;\text{ with }\; \e_i=\pm1, \;\; |k|=|k_1|+\cdots+|k_d|$$
is a completely bounded Fourier $L_p$-multiplier on $\ent^d$ for all $1<p<\8$.  The corresponding multiplier $M_m$, denoted by $LP_c^{Ld}$, is completely bounded on $L_p(\wh\F_\8)$. One can equally consider the more commonly used Littlewood-Paley multiplier:
 $$m=\sum_{(i_1,\cdots, i_d)\in\ent_+^d}\e_i\mathds {1}_{R_{(i_1,\cdots, i_d)}}\,,$$
where $R_{(i_1,\cdots, i_d)}=I_{i_1}\times\cdots\times I_{i_d}$ and $I_{j}=\{k\in\ent: 2^j-1\le |k|<2^{j+1}\}$. It gives rise to a completely bounded multiplier on $L_p(\wh\F_\8)$ too.
 \end{ex}

 %%%%%%%%%%%%%%%%%%%%%%%%%%%%
 \subsection{More paraproducts}\label{More paraproducts}
%%%%%%%%%%%%%%%%%%%%%%%%%%%%
 
We make the paraproducts in Corollary \ref{freepara} more precise in the case of free groups. Let $z=(z_{j,i})_{j\in\nat,\,i\in\nat}\in\T^\nat\times \T^\nat$.  Let $\a_z^{Ld}$ be the  map associated to $(z_{j,i})_{1\le j\le d,i\in\nat}$ in Theorem~\ref{main2}. Note that
 $$\a_z^{Ld}=T_z^{(1)}\cdots T_z^{(d)},$$
 where $T_z^{(j)}(\l(g))=z_{j, i_j}^{k_j}\l(g)$ (with $k_j=0$ for $j>n$) for $g=g_{i_1}^{k_1} \cdots g_{i_n}^{k_n}\in\F_\8$ in reduced form. Put 
 $$\a_z^{Ld, {\rm op}}(x)=[\a_{\overline{z}}^{Ld}(x^*)]^*\;\text{ and }\; T_z^{(j), {\rm op}}(x)=[T_{\overline{z}}^{(j)}(x^*)]^*.$$
Let $\e=(\e_{j,i})_{j\in\nat,i\in\nat}$ be an independent family of symmetric signs. Recall that we use $\mathcal{H}_\e^{(j)}$ and $\mathcal{H}_\e^{(j), {\rm op}}$  to denote the free Hilbert transforms on the (last) $j$-th letters considered in Corollary \ref{letHd},
 $$
H_{\e}^{(j)}(\l_g)= \e _{j,i_j} \l_g\;\text{ and }\; H_{\e}^{(j),{\rm op}}(\l_g)=\e _{j,i_{m+1-j}} \l_g.
 $$
 We use again the convention that  $T_z^{0}=H_\e^0=\a_z^{L0}={\rm Id}$.  For $j, k\ge0$ and $x, y\in\m \com[\F_\8]$, define
 \begin{align*}
 x\,\overset {j,k}\top\,  y&=\E_{z}\,\E_{z'} \,\a_z^{Lj}\,\a_{z'}^{Lk, {\rm op}}\big[\a_z^{Lj}(x)\,\a_{z'}^{Lk, {\rm op}}(y)\big],\\
 x\,\overset {j+,k}\top\,  y&= \E_\e \,H_{\e}^{(j+1)}\big[ H_{\e}^{(j+1)}(x)\overset{j,k}\top  y\big],\\
 x\,\overset{j,k+}\top\,  y&=\E_{\e}\, H_{\e}^{(k+1),{\rm op}}\big[x\overset{j,k}\top  H_{\e}^{(k+1),{\rm op}}(y)\big],
 %x\overset {j+,k+}\top  y&= \E_{\e}\,\E_{\e'}\,H_{\e}^{(j+1)}\,H_{\e'}^{(k+1),{\rm op}}\,[ H_{\e}^{(j+1)}(x)\overset {j,k}\top H_{\e'}^{(k+1),{\rm op}}(y)],\\
% x\,\overset{j-,k}\top\,  y&= x\,\overset {j,k}\top\,  y-  x\,\overset {j+,k}\top\,  y,
 \end{align*}
where $\E_{z}$ and $E_\e$ denote the expectations on $z$ and $\e$, respectively.

As in the setting of free products of von Neumann algebras, we can easily interpret these paraproducts for $x=\l(g)$ and $y=\l(h)$. To this end, we say that the first $j$ blocks  of $g$ survive in $gh$ if  the first $j$ blocks of $gh$ and $g$ are  exactly the same, and that the $j$-th block of $g$ marks in $gh$ if  the $j$-th blocks of $gh$ and $g$ are powers of a same generator. Replacing ``first'' by ``last'' (i.e., counting the letters of a reduced word in the reverse order), we get similar notions.   
   
\smallskip

Thus for $g, h\in\F_\8$
  \begin{itemize}
  \item  $\l(g)\overset{j,k}\top \l(h)=\l(gh)$ if  the first $j$ blocks of $g$ and the last $k$ blocks of $h$ survive in $gh$, $\l(g)\overset{j,k}\top \l(h)=0$ otherwise; 
  \item $\l(g)\overset{j+,k}\top \l(h)=\l(gh)$ if the first $j$ blocks of $g$  and  the last $k$ blocks of   $h$ survive in $gh$, and in addition,  the $(j+1)$-th block of  $g$  marks in $gh$, 
  $\l(g)\overset{j+,k}\top \l(h)=0$ otherwise.
  %\item $\l(g)\overset{j+,k+}\top y=\l(gh)$ if the first $j$ blocks of $g$  and  the last $k$ blocks of   $h$ survive in $gh$, and in addition, both the $(j+1)$-th block of  $g$  and the $(k+1)$-th last block of $h$ mark in $gh$, $\l(g)\overset{j+,k+}\top y=0$ otherwise.
   % \item $x\overset{j-,k}\top y$ if the first $j$ blocks of $g$  and  the last $k$ blocks of   $h$ survive in $gh$, but  the $(j+1)$-th block of $g$ does not mark in $gh$.
  \end{itemize}
  
\smallskip

Our first approach to Theorem \ref{main2} heavily relies on the
boundedness of the above paraproducts and several variants of them.
Now this boundedness immediately
follows from Theorem \ref{main2}.

\begin{prop}\label{freegrouppara}
 All the above paraproducts extend to bounded bilinear maps from $L_{2p}(\wh\F_\8)\times L_{2p}(\wh\F_\8)$ to $L_{p}(\wh\F_\8)$ for all $1<p<\8$ with norm majorized by constants depending only on $p, j$ and $k$.
   \end{prop}

 %%%%%%%%%%%%%%%%%%%%%%%%%%%%
 \subsection{Extension to free products of groups}
%%%%%%%%%%%%%%%%%%%%%%%%%%%%

In the proof of Theorem \ref{main1} one can easily replace $\mathbb Z$
by any abelian discrete group $\Gamma$. However, to go beyond the
abelian case, one needs extra efforts. In  this
subsection, $\G$ will denote a general discrete group. Let
$\G_\infty=\G^{*\nat}$ be the infinite free power of $\G$. Each $g\in
\Gamma_\infty\setminus\{e\}$ is written as a reduced word:
 $$
  g=g_{1}g_{ 2}\cdots g_{n} 
 $$
 with $g_{j}\neq e$ belonging to the $i_j$-th copy of $\G$ in $\G_\8$ and $i_{1}\neq i_{2}\neq\cdots\neq i_n$.  
 
 We begin by extending Theorem~\ref{main2} to this general setting. Define a linear map $\a^{Ld} :\C[\Gamma_\infty] \to\C[\Gamma^d\times \Gamma_\infty]$ as follows: for $g=g_1\cdots g_{n}\in
 \Gamma_\infty$ as above  in reduced form,
\begin{eqnarray*}
 \a^{Ld} (\lambda(g))= \l_{\G^d}(g_1,\cdots, g_{d})\otimes \lambda(g)
 \end{eqnarray*}
with $g_\ell=e$ in $\l_{\G^d}(g_1,\cdots, g_{d})$ if $\ell>n$.  Here we have denoted by $ \l_{\G^d}$ the left regular representation of $\G^d$ to avoid ambiguity ($\l$ being that of $\G_\8$).
 
 \begin{thm}\label{freegammaaction} 
 Let $d\in {\mathbb N}$ and $1<p<\infty$.   Then the map $\a^{Ld}$ extends to a completely isomorphic embedding of $L_p(\widehat\Gamma_\infty)$ into $L_p(\wh\G^d\,\overline\otimes\,\widehat\Gamma_\infty)$.
\end{thm}

 \begin{proof}
  We use an argument similar to that of Theorem \ref{redform}. Let $G=*_{l=1,\cdots,d+1}\Gamma$ and $G_\infty=*_{i\geq1} G$. To avoid confusion we use $\Gamma_{i,l}$ to denote the
 $l$-th copy of $\Gamma$ in the $i$-th copy of $G$ in $G_\infty$.

First, let $\pi: \wh G\to \wh G$ be the $*$-representation given by the cyclic permutation of the copies  that  sends the $l$-th copy of $\Gamma$ to the $l+1\hskip-.15cm\pmod {d+1}$-th copy. Consider the map $T^{Ld}$ on $L_p(\widehat G_\infty)$ given by Theorem \ref{main3} associated to $T_{k}= \pi^{k},1\leq k\leq d$, as well as its inverse associated to $T^{-1}_{k}= \pi^{-k},1\leq k\leq d$.  By Theorem~\ref{main3}, $T^{Ld}$ is a complete isomorphism on $L_p(\widehat G_\infty)$. We will need the restriction of $T^{Ld}$ to a copy of $L_p(\widehat\Gamma_\infty)$ in  $L_p(\widehat G_\infty)$ and will make its presentation more precise. 

Let us identify $\Gamma_\infty$ with $*_{i\geq1}\Gamma_{i,1}$. For an element in $g\in \Gamma_{i,1}$, we denote its copy in $\Gamma_{i,l}$ by $g^{(l)}$. Thus, for  $g=g_1\cdots g_n\in \Gamma_\infty $ in reduced form, we have
 \begin{eqnarray}\label{TLdd}
   T^{Ld}(\lambda(g))=\lambda (g_1^{(2)}\cdots g_d^{(d+1)}g_{d+1}^{(1)}\cdots g_n^{(1)}).
 \end{eqnarray}  As we have explained, (\ref{TLdd}) defines a complete isomorphic embedding of $L_p(\widehat\Gamma_\infty)$ into $L_p(\widehat G_\infty)$.

Next, we consider the group morphism $\phi$ from $G_\infty$ onto $\Gamma^d$  such that for all $i$ and $g\in \Gamma_{i,k}\subset G_\infty$
 $$\phi(g)=e\;\text{ if }k=1\;\text{ and }\; \phi(g)=(\underbrace{e,\cdots,e}_{k-2},g,e,\cdots, e) \;\text{ if }2\le k\le d+1.$$
Let  $U=\l_{\Gamma^d}\circ \phi$. Then $U$ is a unitary representation of $G_\infty$ on $\ell_2(\Gamma^d)$.  Applying the Fell absorption principle to $U$, we get a completely isometric embedding of $L_p(\widehat G_\infty)$ into $L_p(\wh{\G^d}\,\overline{\tens}\,\widehat G_\infty)$:
 \begin{eqnarray}\label{fell}
  \Pi_\phi: \sum_{g\in G_\infty} c(g)\l(g)\mapsto \sum_{g\in G_\infty}  c(g)\l_{\Gamma^d}(\phi(g))\otimes \l(g).
  \end{eqnarray}
By (\ref{TLdd}) and (\ref{fell}),  we see that
 $$\a^{Ld}(x)=\Pi_\phi\cdot T^{Ld}(x),\quad x\in \C[\Gamma_\infty].$$ 
Thus  $\a^{Ld}$ extends to a completely isomorphic embedding on $L_p(\widehat\Gamma_\infty)$.
 \end{proof}

\begin{rk}
    There is an  alternate proof to Theorem~\ref{freegammaaction}. First, one may apply  Theorem \ref{redform} and Remark \ref{explicit} for 
$\A_i=\wh{\Gamma^d}\,\overline\tens\, \widehat{\Gamma}$ for all $i\in {\mathbb N}$, and $\B=\widehat{\Gamma^d}\,\tens \C$,  and check that 
    $$\|\a^{L1}(x)\|_{L_p(\widehat  {\Gamma^d}\,\overline\otimes \,\widehat \Gamma_\infty)}\simeq_{p}\|x\|_{L_p(\widehat \Gamma_\infty)}$$
for any $x\in\com[\Gamma_\infty]$. Next, one may apply Theorem \ref{redform} repeatedly to get 
  \begin{eqnarray*}\label{last}
  \|\a^{Ld }(x)\|_{L_p(\widehat  {\Gamma^d}\,\overline\otimes \,\widehat \Gamma_\infty)}
  \simeq_{p}\|\a^{L(d-1)}(x)\|_{L_p(\widehat  {\Gamma^d}\,\overline\otimes \,\widehat \Gamma_\infty)}
  \simeq_{p}\cdots \simeq_{p}\|x\|_{L_p(\widehat\Gamma_\infty)}.
 \end{eqnarray*}
 This alternate proof gives a better constant.
\end{rk}

As in the free group case, the previous theorem, together with transference, easily implies a multiplier result on $\G_\8$. Given a bounded function $m$ on ${\Gamma}^d$, define  a linear map $M_m$ on
 $L_2(\hat{\Gamma}_\infty)$ by
  $$M_m(\l(g))=m(g_1,g_2,\cdots,g_d)\l (g)$$
with $g_\ell=e$ in $m(g_1,g_2,\cdots,g_d)$ if $\ell>n$ for every $g\in\G_\8$ in reduced form.   Note that $\wh\G$ is QWEP iff $\G$ is hyperlinear (cf. \cite{Ra08}).

\begin{thm}\label{freegamma} 
  Assume that $\wh\G$ is QWEP. Let $d\in {\mathbb N}$ and $1<p<\infty$.   If the Fourier multiplier $T_m$ is completely bounded on $L_p(\wh{\G^d})$, then  $M_m$ extends to a completely bounded map on $L_p(\wh\G_\infty)$.
\end{thm}

 \begin{proof}
 This proof is similar to that of Theorem~\ref{main1}. As $\widehat \Gamma$ is QWEP, so is $\widehat\Gamma_\infty$. Thus $T_m\tens{\rm Id}_{L_p(\wh\G_\8)}$ is completely bounded on $L_p(\wh{\G^d}\,\overline\tens\,\wh\G_\8)$ by Junge's noncommutative Fubini theorem \cite{ju-fubini}. Let $x\in\com[\G_\8]$. Using the action $\a^{Ld}$ in Theorem~\ref{freegammaaction}, we have
  $$\|x\|_{L_p(\wh\G_\8)}\simeq_{d,p}\|\a^{Ld}(x)\|_{L_p(\wh{\G^d}\,\overline\tens\,\wh\G_\8)}\,.$$
It remains to use the interviewing formula
 $$T_m\tens{\rm Id}[\a^{Ld}(x)]=a^{Ld}[M_m(x)]$$
to conclude as in the proof of Theorem~\ref{main1}.
 \end{proof}

In the same line,  we conclude this subsection by stating  another application. Consider a family of discrete groups $\Gamma_i$, $i\in I$,  and its free product $\Gamma_\infty=*_{i\in I} \Gamma_i$; consider also  a  family of finite von Neumann algebras $(\M_i,\tau_i), i\in I$, and its von Neumann tensor product $(\M,\tau)=\overline\tens_{i\in I}(\M_i,\tau_i)$. Let  $\{m_{i,g}\}_{g\in\G_i}\subset \M_i$ for all $i\in I$, and let $M_i$ be the operator-valued Fourier multiplier on $\Gamma_i$:
 $$M_i (\lambda(g))=m_{i,g}\otimes\lambda (g),\quad g\in\G_i.$$ 
We construct a map $M^{Ld}$ similar to Theorem \ref{freegamma}. Given $g=g_{1}\cdots g_{n}\in \G_\8$ in reduced form, define
  $$ M^{Ld}(\l(g))= \left\{ \begin{array}{ll}
       m_{i_1,g_1}\otimes  \cdots \otimes m_{i_d,g_d}\otimes \l(g) & \textrm{ if }n\geq d,\\
     m_{i_1,g_1}\otimes  \cdots \tens m_{i_n,g_n}  \otimes\lambda(g) & \textrm{ if } n< d.\end{array}\right.$$
 
 \begin{thm}\label{multd}
  Let $1<p<\infty$. Then $M^{Ld}$ extends to a completely bounded Fourier multiplier from $L_p(\wh\Gamma_\8)$ to $L_p(\M\,\overline\otimes\,\wh\Gamma_\8)$ iff the family 
  $\big\{\big\|M_i\big\|_{{\rm cb}\big(L_p(\wh\Gamma_i),\, L_p(\M_i\,\overline\otimes\, \wh\Gamma_i)\big)}\big\}_{i\in I}$ is bounded. In this case, we have
   $$\big\|M^{Ld}\|_{{\rm cb}\big(L_p(\widehat\Gamma_\infty),\,L_p(\M\, \overline\otimes\,\widehat\Gamma_\infty)\big)}
   \simeq_{p,d} \sup_{i\in I} \big\|M_i\big\|_{{\rm cb}\big(L_p(\wh\Gamma_i),\, L_p(\M_i\,\overline\otimes \,\wh\Gamma_i)\big)}.$$
    \end{thm}

 \begin{proof}
  We can easily adapt the proofs of Theorem~\ref{freegammaaction}  to the present setting, so we omit the details.
 \end{proof}

%%%%%%%%%%%%%%%%%%%%%%%%%%%%%%%%%%%%%%%%%%%%%%%%%%%%%%
%%%%%%%%%%%%%%%%%%%%%%%%%%%%%%%%%%%%%%%%%%%%%%%%%%%%%%

\appendix
\section{Endpoint boundedness of free Hilbert transforms}

%%%%%%%%%%%%%%%%%%%%%%%%%%%%%%%%%%%%%%%%%%%%%%%%%%%%%%
%%%%%%%%%%%%%%%%%%%%%%%%%%%%%%%%%%%%%%%%%%%%%%%%%%%%%%

Our arguments rely on the $L_p$-boundedness ($1<p<\8$) of the free Hilbert transforms $H_\e$ in Lemma \ref{MR17} and their variants $H_\e^{(j)}$ in Corollary~\ref{letHd}.  We will discuss their bounds on homogeneous polynomials when $p=\infty$, since they cannot be bounded in full generality at the end point. It is also natural to ask if one
can get an $L_\infty$-BMO boundedness for the BMO spaces studied in \cite{JM12}.
 
  %%%%%%%%%%%%%%%%%%%%%%%%%%%%
  \subsection{Bounds on homogeneous polynomials}
   %%%%%%%%%%%%%%%%%%%%%%%%%%%%
 
We work in the setting of amalgamated free products of von Neumann algebras as in section~\ref{Multipliers on amalgamated  free products}.   Let $H_\e^{(j)}$ be the maps introduced  in Corollary~\ref{letHd} with $\e$ a family of signs.

 The case  $j=1$ of the following theorem  follows from \cite[Proposition 2.8]{JPX07}.
   
  \begin{thm} 
   Let $d\ge1$ and $1\le j\le d$. Then for any $x\in\m W_d$ we have
  $$
 \|H_\e^{(j)}(x) \|_\infty\lesssim  \min\big\{\log(j+2),\,\log (d-j+2)\big\}\|x\|_\infty.
 $$
  \end{thm}
 
\begin{proof}  For $z\in \mathbb T$, let $U_z$ be the unitary on $L_2({\mathcal A})$ sending   $w\in{\mathcal W}_n$ to $z^nw$ for $n\ge0$. Given $x\in {\mathcal W}_d,y \in L_2({\mathcal A})$ and $0\leq k\leq 2d$, let 
 $$x\bot^{k} y=\E_z [z^{d-k}U_{\overline{z}}(xU_z(y))].$$
 It  is easy to verify that
  $$xy=\sum_{k=0}^{2d} x\bot^{k}y \;\;\text {and }\;\; \|x\bot^{k} y\|_2\leq \|x\|_\infty\|y\|_2.$$
Moreover, we easily show 
 $$H_\e^{(j)}(x)\bot^{k}y
  = \left\{\begin{array}{lcl}
  H_\e^{(j)}(x\bot^{k}y) & \textrm{if}& k\leq 2(d-j),\\
  x\bot^{k} H_\e^{(d+1-j)}(y)& \textrm{if}& k> 2(d-j).\end{array}\right. $$
Thus
  $$H_\e^{(j)}(x) y=\sum_{0\leq k\leq 2(d-j)}H_\e^{(j)}(x\bot^{k}y)+\sum_{ 2(d-j)< k\leq 2d}x\bot^{k} H_\e^{(d+1-j)}(y).$$
  On the other hand, using the elementary estimate
 $$\big\|\sum_{k=0}^n z^k\big\|_{L_1(\T)}\simeq \log(n+2),$$
 we get
  $$\big\|\sum_{k=0}^{n}x\bot^{k} y\big\|_2 + \big\|\sum_{k=n}^{2d}x\bot^{k} y\big\|_2\lesssim \min\big\{\log (n+2),\,\log (2d-n+2)\big\}  \|x\|_\infty\|y\|_2. $$
Therefore, since  $H_\e^{(j)}$ is isometric on $L_2(\A)$, we deduce
 \begin{align*}
 \big\|H_\e^{(j)}(x) y\big\|_2
 &\leq \big\|\sum_{k=0}^{ 2(d-j)}x\bot^{k}y\big\|_2+\big\|\sum_{k=2(d-j)+1}^{ 2d}x\bot^{k}H_\e^{(d+1-j)}(y)\big\|_2\\
 &\lesssim\min\big\{\log(j+2),\,\log (d-j+2)\big\}\|x\|_\infty\|y\|_2,
\end{align*}
whence  the desired estimate on $\|H_\e^{(j)}(x)\|_\8$.
\end{proof}

 %%%%%%%%%%%%%%%%%%%%%%%%%%%%
 \subsection{Failure of the $L_\infty$-BMO boundedness}
  %%%%%%%%%%%%%%%%%%%%%%%%%%%%
  
Let us  restrict ourselves to the free group case.  Recall that  the Poisson semigroup $(S_t)_{t\geq 0}$ on $\F_\infty$ is the normal unital completely positive semigroup given by
  $$S_t(\lambda(g))=e^{-t|g|}\lambda(g),\quad g\in\F_\infty.$$
We also recall the definitions  of various ${\rm BMO}$-spaces according to \cite{JM12}. As usual we denote by $L_p^0(\wh \F_\infty)$ the subspace of centered elements  (i.e., elements with vanishing trace) in  $L_p(\wh\F_\infty)$. Define
 \begin{align*}
{\rm BMO}^c(S)&=\big\{x\in L^0_2(\wh\F_\infty): \| x\|_{{\rm BMO}_{S}} <\infty\big\}, \\
 {\rm bmo}^c(S)&=\big\{x\in L^0_2(\wh\F_\infty): \| x\|_{{\rm bmo}_{S}}<\infty\big\},
\end{align*}
where
 \begin{align*}
 \| x\|_{{\rm BMO}^c(S)}
 &= \sup_{t>0} \big\|S_t\big[|x- S_t(x)|^2\big]\big\|_\infty^{1/2},\\
 \| x\|_{{\rm bmo}^c(S)}
 &= \sup_{t>0} \big\|S_t(|x|^2)-|S_t(x)|^2)\big\|_\infty^{1/2}.
 \end{align*}
Similarly, we define the row versions ${\rm BMO}^r(S)$ and ${\rm bmo}^r(S)$ by passing to adjoints. One of the main results of \cite{JM12} states that the intersection space ${\rm BMO}^c(S)\cap {\rm BMO}^r(S)$ behaves well with complex interpolation, i.e., it replaces $L_\8$ as an endpoint space in the complex interpolation scale $\big\{L_p(\wh\F_\8)\big\}_{p>1}$.

\medskip

Let $H_\e$ denote the free Hilbert transform of the first letters associated to a sequence of signs, see Lemma~\ref{MR17}. It is easy to see that $\|H_\varepsilon(x)\|_{{\rm bmo}^c(S)}=\|x\|_{{\rm bmo}^c(S)}$ for $x\in \mathbb C[\F_\infty]$. We will explain  why one cannot hope the boundedness of $H_\e$ from $L_\infty(\wh\F_\8)=\wh\F_\8$ to  any of  ${\rm BMO}^r(S)$, ${\rm BMO}^c(S)$ or ${\rm bmo}^r(S)$.

 \begin{lemma}
Let  $a,b\in\F_\8$ be two free elements. Let $z$ be a finite sum $z=\sum_{k\ge1} c_k\lambda_{a^k}$ and $z_{n}=z\l(b^{n})$. Then
\begin{enumerate}[\rm i)]
\item $\displaystyle \lim_{n\to \infty} \|z_n\|_{{\rm bmo}^r(S)}=\|z\|_\infty$,
\item $\displaystyle (e^{-1}-e^{-2})\|z\|_\infty\leq \limsup_{n\to \infty} \|  x_n\|_{{\rm BMO}^\a(S)}\leq 2\|z\|_\infty,\; \a\in\{c, r\}.$ 
\end{enumerate}
 \end{lemma}
 
 \begin{proof}
i) It is clear that 
 $$\|z_n\|_{{\rm bmo}^r(S)}\leq \|z_n\|_\infty=\|z\|_\infty.$$ 
For a fixed $t>0$, we have $\lim \|S_t(z_n)\|_\infty=0$. Thus, 
 $$\limsup_{n\to\8} \| z_n\|_{{\rm bmo}^r(S)}^2\geq \sup_t \big\|S_t(z_nz_n^*)\big\|_\infty=\|z\|_\infty^2.$$
   
ii) The upper bound is clear. For the lower, we use the Kadison inequality to ensure
      $$\big\|S_t|z_n^*-S_t(z_n^*)|^2\big\|_\8, \;\big\|S_t|z_n-S_t(z_n)|^2\big\|_\8\geq \big\| (S_t-S_{2t})z_n\big\|_\8^2.$$ 
But in $\wt\F_\infty$,
 $$\lim_{n\to\infty} \big((S_{\frac1n}-S_{\frac2n})z_n\big)\lambda_{b^{-n}}= (e^{-1}-e^{-2})z.$$
This finishes the proof.
 \end{proof}
  
As $H_\varepsilon(z_n)=H_\varepsilon(z)\lambda_{b^n}$, we get
\begin{corollary}
 The transform $H_\varepsilon$ is unbounded from $L_\infty(\wh \F_\8)$ to any of  ${\rm BMO}^r(S)$, ${\rm BMO}^c(S)$ and ${\rm bmo}^r(S)$.
\end{corollary}

\bigskip

{\bf Acknowledgements.}  The first author is partially supported by NSF award DMS-1700171. The second and third authors are partially supported by the French ANR project No. ANR-19-CE40-0002, and the  third author is also partially supported by  the Natural Science Foundation of China (No.12031004).

\bigskip

%%%%%%%%%%%%%%%%%%%%%%%%%%%%%%%%%%%%%%%%%%%%%%%%%%%%%%
%%%%%%%%%%%%%%%%%%%%%%%%%%%%%%%%%%%%%%%%%%%%%%%%%%%%%%

\end{document}